\documentclass{amsart}
\usepackage{amsthm}
\usepackage{pinlabel}
\usepackage{enumitem}
\usepackage{color}
\newlist{steps}{enumerate}{1}
\setlist[steps, 1]{label = Step \arabic*:}
\newtheorem{theorem}{\bf Theorem}[section]
\newtheorem{lemma}[theorem]{\bf Lemma}
\newtheorem{definition}[theorem]{\bf Definition}
\newtheorem{corollary}[theorem]{\bf Corollary}
\newtheorem{proposition}[theorem]{\bf Proposition}
\newtheorem{remark}[theorem]{\bf Remark}
\newcommand{\rme}{\mathrm{e}}
\newcommand{\rmi}{\mathrm{i}}

\newcommand{\defeq}{\mathrel{\mathop:}=}

\begin{document}

\title{Knotted surfaces as vanishing sets of polynomials}

\author{Benjamin \textsc{Bode}}
\address{Department of Mathematics, Graduate School of Science, Osaka University, Toyonaka, Osaka 560-0043, Japan}
\email{ben.bode.2013@my.bristol.ac.uk}

\author{Seiichi \textsc{Kamada}}
\address{Department of Mathematics, Graduate School of Science, Osaka University, Toyonaka, Osaka 560-0043, Japan}
\email{kamada@math.sci.osaka-u.ac.jp}




\begin{abstract}
We present an algorithm that takes as input any element $B$ of the loop braid group and constructs a polynomial $f:\mathbb{R}^5\to\mathbb{R}^2$ such that the intersection of the vanishing set of $f$ and the unit 4-sphere contains the closure of $B$. The polynomials can be used to create real analytic time-dependent vector fields with zero divergence and closed flow lines that move as prescribed by $B$. We also show how a family of surface braids in $\mathbb{C}\times S^1\times S^1$ without branch points can be constructed as the vanishing set of a holomorphic polynomial $f:\mathbb{C}^3\to\mathbb{C}$ on $\mathbb{C}\times S^1\times S^1\subset\mathbb{C}^3$. Both constructions allow us to give upper bounds on the degree of the polynomials.
\end{abstract}

\maketitle

\section{Introduction}\label{sec:intro}
For a (polynomial) function $f:\mathbb{R}^m\to\mathbb{R}^n$, where $m$ and $n$ are positive integers, we call the set of points where $f$ vanishes, i.e.,
\begin{equation}
V_f\defeq f^{-1}(0,0,\ldots,0)=\{(x_1,x_2,\ldots,x_m)\in\mathbb{R}^m:f(x_1,x_2,\ldots,x_m)=(0,0,\ldots,0)\},
\end{equation}
the \textit{vanishing set} of $f$. If $f$ appears in the context of a physical system, for example as a quantum wavefunction, we may also refer to $V_f$ as the \textit{nodal set} of $f$. We often write $f^{-1}(0)$ for the vanishing set instead of $f^{-1}(0,0,\ldots,0)$ or $V_f$.

The best-known connection between knots and vanishing sets of polynomials is Milnor's study of algebraic links, i.e., links of isolated singularities of complex plane curves \cite{milnor}. However, singularities are not necessary to establish interesting relations between properties of the polynomials and topological properties of the corresponding links. The set of transverse $\mathbb{C}$-links, introduced by Rudolph as the links that arise as transverse intersections of complex plane curves and the unit 3-sphere \cite{rudolph}, for example have been shown to be exactly the set of quasipositive braid closures \cite{boileau, rudolph83, rudolph84, rudolph:quasi}.

More recently, polynomials with knotted vanishing sets have also caught the eye of theoretical physicists who are interested in the explicit construction of knotted configurations in physical systems. Such a configuration is usually described by a function from 3-dimensional space to some target space. One example of this would be a quantum wavefunction that takes values in $\mathbb{C}$. Berry found that there are eigenstates $\Psi:\mathbb{R}^3\to\mathbb{C}$ of the hydrogen atom, for which the nodal set of $\Psi$ is knotted \cite{berry}. Berry's problem on knotted nodal sets of quantum systems was resolved for the harmonic oscillator \cite{daniel:berry} and the hydrogen atom \cite{daniel:coulomb}. Examples from other areas in theoretical physics abound \cite{mark, km:2016topology, kedia1, kauffman, ma:2014knotted, ma:2016global, sutcliffe}.

Polynomials whose vanishing sets are in some sense knotted, for example in the form of $f:\mathbb{R}^4\to\mathbb{R}^2$ with $f^{-1}(0)\cap S^3$ being knotted, can potentially be used in many areas of theoretical physics to construct such knotted configurations. 
In \cite{bode:2016polynomial} Dennis and the first author developed an algorithm that constructs such a polynomial for any given link. Furthermore, these polynomials, written as maps $\mathbb{C}\times\mathbb{R}^2\to\mathbb{C}$, can be taken to be holomorphic in the complex variable and of relatively low degree.

Usually, a physical system is not described by a polynomial. Quantum wavefunctions for example have to be normalisable and be the eigenfunction of some Hermitian operator. Hence we need some procedure to turn the constructed polynomials into physically meaningful functions. In some areas this is more established than in others.

So far, we have considered a configuration of a physical system as a function from 3-dimensional space. This means that it describes the state of the system at a fixed time $t$. The algorithm in \cite{bode:2016polynomial} allows us to construct such functions for any link and for several physical systems \cite{bode:2016lemniscate}. However, the configurations that are constructed in this way cannot be expected to be stable, i.e., by construction the field $\Psi_t$ contains the desired link at the time $t$, say $\Psi_t^{-1}(0)=L$ for example, but as the field evolves according to some differential equation, the link is moved around and potentially changed or removed altogether. Hence at a later time $t'$ we have absolutely no guarantee that the field $\Psi_{t'}$ still contains the desired link. We cannot expect $\Psi_{t'}^{-1}(0)$ to be ambient isotopic to $L$.

In several areas of theoretical physics there are now rigorous results regarding the existence and construction of knotted solutions. Apart from the solutions to Berry's problem these have been proven to exist in monochromatic waves (satisfying the Helmholtz equation) \cite{daniel:submani} and in the context of magnetic fields induced by knotted wires \cite{ulam} and of knotted stream and vortex lines in fluid mechanics \cite{daniel:navier, daniel:steady, daniel:vortexEuler}. These results are to a large part constructive. In particular, explicit solutions can be found numerically with the aid of a computer.

The goal of this article is to extend the construction from \cite{bode:2016polynomial} to higher dimensions, i.e., knotted surfaces in 4 dimensions. Since the original construction is based on braids, we consider different generalisations of the concept of a braid: the loop braid group and surface braids in $\mathbb{C}\times S^1\times S^1$.
\begin{theorem}
\label{thm:intro1}
There is an algorithm that constructs for every given element $B$ of the loop braid group a polynomial $f:\mathbb{R}^5\to\mathbb{R}^2$ such that the vanishing set $f^{-1}(0)\cap S^4$ on the unit 4-sphere contains a set that is ambient isotopic to the closure of $B$. The algorithm provides an upper bound on the degree of $f$ in terms of the number of strands and the number of crossings of $B$.
\end{theorem}

By composing the constructed polynomial with an inverse stereographic projection from $\mathbb{R}^4$ to $S^4$ we obtain a polynomial from $\mathbb{R}^4$ to $\mathbb{R}^2$ whose vanishing set contains the closure of $B$. The existence of such polynomials is well known. In fact, the Nash-Tognoli Thereom (cf. Theorem 14.1.4 in \cite{bochnak:rag}) guarantees that there is always a polynomial from $\mathbb{R}^4$ to $\mathbb{R}^2$ whose vanishing set is ambient isotopic to (as opposed to `contains') the closure of $B$, or any given finite collection of embedded tori in $\mathbb{R}^4$ for that matter. Additionally, it follows from \cite{danielhector} that there is an analytic submersion with this property. We would like to point out that in general zero is not necessarily a regular value of the polynomials in Theorem \ref{thm:intro1}. However, the points on $B$ can be taken to be regular points. That is, if the vanishing set contains critical points, then they lie on the extra components. At this moment we are not aware of a constructive method that would remove the extra components or guarantee that zero becomes a regular value.

\begin{theorem}
\label{thm:intro2}
There is an algorithm that constructs for any spinning braid $B$ (cf. Definition \ref{def:spinning}) in $\mathbb{C}\times S^1\times S^1$ a holomorphic polynomial $f:\mathbb{C}^3\to\mathbb{C}$ such that the vanishing set $f^{-1}(0)\cap (\mathbb{C}\times S^1\times S^1)$ is ambient isotopic to $B$. The algorithm provides an upper bound on the degree of $f$.
\end{theorem}

The precise upper bounds on the polynomial degrees are given in Propositions \ref{prop:bound} and \ref{prop:bound2}, respectively.
The motivation for these constrcuctions is threefold. Firstly, we hope to find similar relations between properties of the polynomials and topological features of the surfaces as in the lower-dimensional case, making this a worthwhile endeavour in the intersection of classical real algebraic geometry and the topology of knotted surfaces. Indeed there is, like for 1-dimensional links in $S^3$, an upper bound on the degree of the constructed polynomials in terms of the number of strands and the number of crossings of the braid used for the construction. Secondly, the theorems can be seen as constructive counterparts to the existence results in the vein of Nash and Tognoli. Thirdly, the new algorithm has the potential of providing us with a tool to create knotted field configurations, whose time evolution is in a topological sense controlled. This goes beyond the stability condition alluded to above that demands that the field contains a given link for all time. Since the algorithm from Theorem \ref{thm:intro1} works for any loop braid, we can control how the different components of the link twist around each other as time evolves. We would like to emphasize that at this point the algorithm does not take any differential equations into account, so that at the moment hoping for such applications seems very optimistic. However, if solutions can be constructed from polynomials, there is the possibility that topological and algebro-geometric properties are reflected in physical quantities, such as in \cite{weaving}, where the helicity of an electromagnetic field with knotted field lines is related to the degree of a polynomial, which in turn is related to the number of strands of a braid.

At this stage, we still lack a proper procedure to turn our polynomials into physically meaningful functions. However, we point to some promising observations in this regard. In particular, we find.
\begin{proposition}
\label{prop:vector}
Let $B\subset\mathbb{R}^3\times[0,2\pi]$ be a loop braid. Then we can construct a time-dependent real analytic vector field $V_t:\mathbb{R}^3\to\mathbb{R}^3$, $t\in\mathbb{R}$, such that
\begin{itemize}
\item $\nabla\cdot V_t=0$ for all $t\in\mathbb{R}$,
\item $V_t=V_{t+2\pi}$ for all $t\in\mathbb{R}$,
\item There are closed flow lines of $V_t$ that move as prescribed by $B$ as $t$ varies between $0$ and $2\pi$, i.e. for all $t\in\mathbb{R}$ the field $V_t$ is tangent to $B\cap\left(\mathbb{R}^3\times\{t\text{ mod }2\pi\}\right)$.
\end{itemize}
\end{proposition}
This result can be seen as a first small step in the construction of time-dependent magnetic fields with closed flow lines, whose time evolution is determined by a loop braid $B$.

While Theorems \ref{thm:intro1} and \ref{thm:intro2} refer to an explicit constructive procedure, we also assert the existence of certain functions without a procedure to find them yet.
\begin{theorem}
\label{prop:holosurf1}
Let $B$ be a surface braid in $\mathbb{C}\times S^1\times S^1$ of degree $m$ and without any branch points. Then there exists a holomorphic polynomial $f:\mathbb{C}^3\to\mathbb{C}$ such that $f^{-1}(0)\cap(\mathbb{C}\times S^1\times S^1)$ is equivalent to $B$ and $\deg_u f=m$, where $\deg_u f$ is the polynomial degree of $f$ with respect to the first complex variable $u$.
\end{theorem}
Definitions of surface braids and the notion of equivalence between surface braids are given in Section \ref{sec:2dim}.






The remainder of this paper is structured as follows. Section \ref{sec:review} reviews the construction of 1-dimensional links $L$ as vanishing sets of polynomials $f:\mathbb{R}^4\to\mathbb{R}^2$, $f^{-1}(0)\cap S^3=L$. In Section \ref{sec:loop} we describe the algorithm for Theorem \ref{thm:intro1}, which allows us to construct arbitrary loop braid closures as subsets of vanishing sets of polynomials $f:\mathbb{R}^5\to\mathbb{R}^2$ on $S^4$. We also prove a bound on the polynomial degree and discuss possible applications in theoretical physics. Moreover, we show that Theorem \ref{thm:intro1} can be extended to certain elements of other motion groups of split links and prove Proposition \ref{prop:vector}. Section \ref{sec:2dim} introduces the algorithm for Theorem \ref{thm:intro2}, which again leads to a bound on the degree of the constructed polynomials, and proves Theorem \ref{prop:holosurf1}. 

This work is supported by JSPS KAKENHI Grant Numbers JP18F18751 and JP19H01788. The first author is supported as a JSPS International Research Fellow. The authors would like to thank Daniel Peralta-Salas and an anonymous referee for valuable comments.

\section{Review of the classical case}
\label{sec:review}
In \cite{bode:2016polynomial} we outlined an algorithm that constructs for any given link type $L$ in $S^3$ a polynomial $f:\mathbb{R}^4\to\mathbb{R}^2$ such that $f^{-1}(0)\cap S^3=L$. This construction takes as an input a braid that closes to $L$ and can be performed so that we obtain information about the degree and the number of argument-critical points of $f$ on $S^3$ from the input braid word. In this section we review the algorithm given in \cite{bode:2016polynomial}, which should make the generalisation to higher-dimensions in the subsequent sections easier to follow.

Let $B$ be a braid on $s$ strands that closes to the given link $L$. We denote the set of link components of the closure of $B$ by $\mathfrak{C}$ and the number of strands that form a component $C\in\mathfrak{C}$ by $s_C$. Then a parametrisation of $B$ in $\mathbb{R}^2\times[0,2\pi]$ is given by
\begin{equation}
\label{eq:para}
\bigcup_{C\in\mathfrak{C}}\bigcup_{j=1}^{s_C}\left\{\left(X_{C,j}(t),Y_{C,j}(t),t\right)|t\in[0,2\pi]\right\},
\end{equation}
where $X_{C,j}$ and $Y_{C,j}$ are smooth functions from $[0,2\pi]$ to $\mathbb{R}$ such that for any $t\in[0,2\pi]$ $(X_{C,j}(t),Y_{C,j}(t))=(X_{C',j'}(t),Y_{C',j'}(t))$ implies $C=C'$ and $j=j'$ and such that for every $(C,j)$ we have $(X_{C,j}(2\pi),Y_{C,j}(2\pi))=(X_{C,j+1}(0),Y_{C,j+1}(0))$, where the index $j$ is taken $\text{ mod }s_C$.

Given a braid parametrisation as in Eq. (\ref{eq:para}) we can define a family of functions $g_{\lambda}:\mathbb{C}\times[0,2\pi]\to\mathbb{C}$ via
\begin{equation}
\label{eq:g}
g_{\lambda}(u,t)=\prod_{C\in\mathfrak{C}}\prod_{j=1}^{s_C}\left(u-\lambda(X_{C,j}(t)+\rmi Y_{C,j}(t))\right).
\end{equation}
For every choice of the real parameter $\lambda>0$ the vanishing set of $g_{\lambda}$ is the braid $B$ as in the parametrisation in Eq. (\ref{eq:para}), but scaled in the $x$- and $y$-coordinates by $\lambda$.

If for every $C$ there are trigonometric polynomials $F_C$ and $G_C:[0,2\pi]\to\mathbb{R}$ such that
\begin{equation}
\label{eq:trigo}
(X_{C,j}(t),Y_{C,j}(t))=\left(F_{C}\left(\frac{t+2\pi (j-1)}{s_C}\right),G_{C}\left(\frac{t+2\pi (j-1)}{s_C}\right)\right),
\end{equation}
then the function $g_{\lambda}$ is not only a polynomial in the complex variable $u$, but also in $\rme^{\rmi t}$ and $\rme^{-\rmi t}$ once the product in Eq. (\ref{eq:g}) is expanded \cite{bode:2016polynomial}.

This means that there is a polynomial $f_{\lambda}:\mathbb{C}^2\to\mathbb{C}$ in complex variables $u$, $v$ and the complex conjugate $\overline{v}$ such that $f_{\lambda}(u,\rme^{\rmi t})=g_{\lambda}(u,t)$. We obtain $f_{\lambda}$ from $g_{\lambda}$ by replacing every instance of $\rme^{\rmi t}$ in the expanded product of Eq. (\ref{eq:g}) by $v$ and every instance of $\rme^{-\rmi t}$ by the conjugate $\overline{v}$. We call such functions that are holomorphic with respect to one of the complex variables but not with respect to the other \textit{semiholomorphic}. We show in \cite{bode:2016polynomial} that for small enough values of $\lambda>0$ the intersection $f_{\lambda}^{-1}(0)\cap S^3$ is isotopic to $L$, the closure of $B$.

We can now summarize the algorithm in \cite{bode:2016polynomial} as follows:
\ \\

\noindent
\textbf{Algorithm 0:}
\begin{steps}
\item From the given braid word find the trigonometric polynomials $F_C$.
\item Next find the trigonometric polynomials $G_C$. This results in a braid parametrisation as in Eq. (\ref{eq:para}) and Eq. (\ref{eq:trigo})
\item Define $g_{\lambda}$, write it as a polynomial by expanding the product and define $f_{\lambda}$.
\item Determine how small $\lambda$ has to be chosen.
\end{steps}

\begin{theorem}
(cf. \cite{bode:2016polynomial}) For every braid $B$ on $s=\sum_C s_C$ strands with a diagram with $\ell$ crossings Algorithm 0 constructs a semiholomorphic polynomial $f:\mathbb{R}^4\to\mathbb{R}^2$ such that the vanishing set $f^{-1}(0)\cap S^3$ on the unit 3-sphere is ambient isotopic to the closure of $B$. Furthermore, we have 
\begin{equation}
\deg f\leq \sum_C\max\left\{\left\lfloor\frac{(s_C+1)(s_C\ell-1)+\ell s_C(s-s_C)-1}{2}\right\rfloor,s_C\right\}.
\end{equation}
\end{theorem}

The only part of this algorithm that requires some additional explanation and is not computationally trivial consists of the first two steps, the construction of the trigonometric polynomials $F_C$ and $G_C$. Both can be obtained by trigonometric interpolation. There are other ways to do this as well, but this method ensures a relatively low degree.

A detailed description of the procedure can be found in \cite{bode:2016polynomial}. Here we want to give an example of how the interpolation works to highlight the key points of the construction that we generalize in the subsequent sections.

\begin{figure}
\centering
\labellist
\large
\pinlabel a) at 10 500
\endlabellist
\includegraphics[height=3cm]{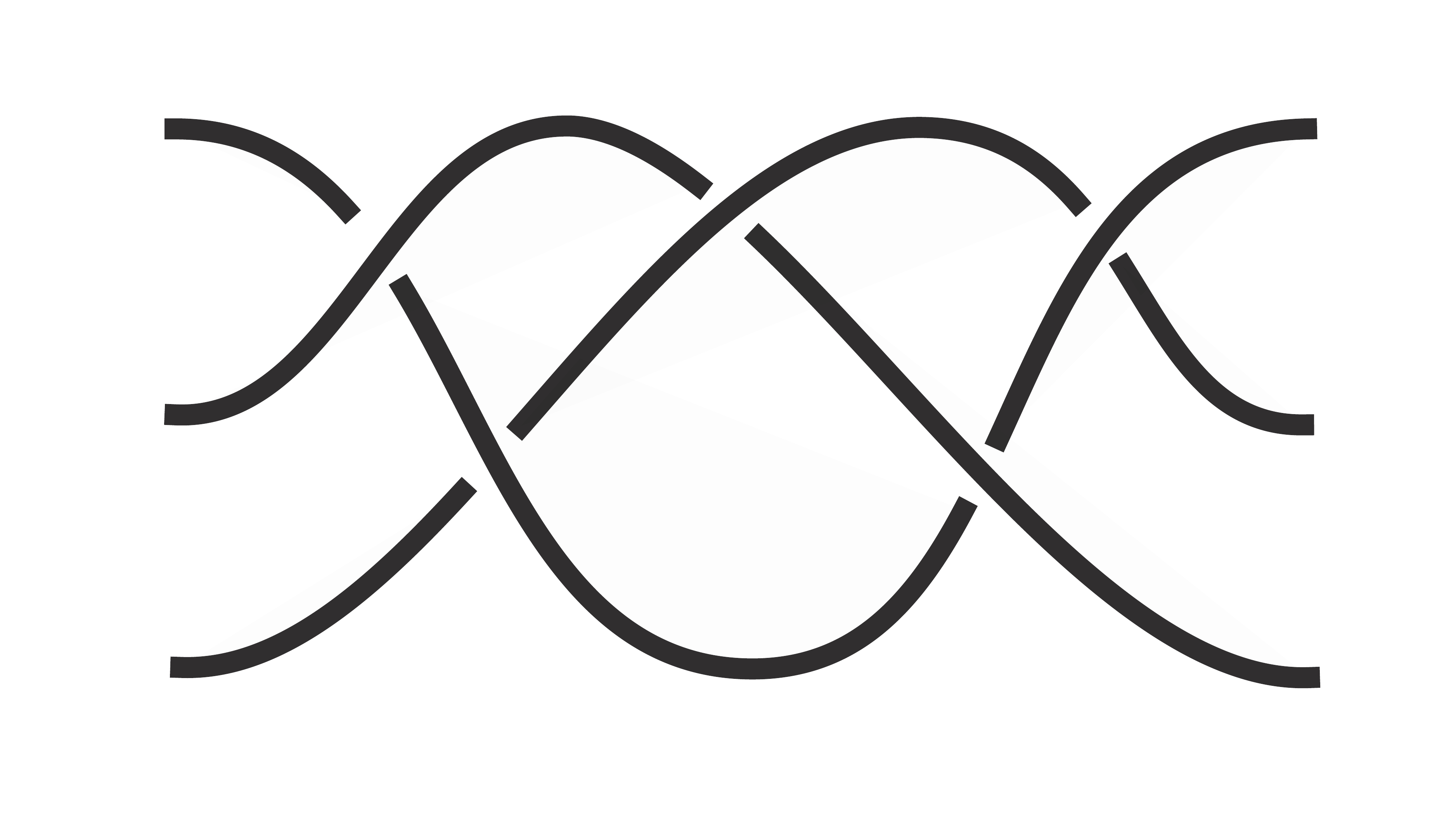}\\
\vspace{0.5cm}
\labellist
\large
\pinlabel b) at 0 130
\endlabellist
\includegraphics[height=2cm]{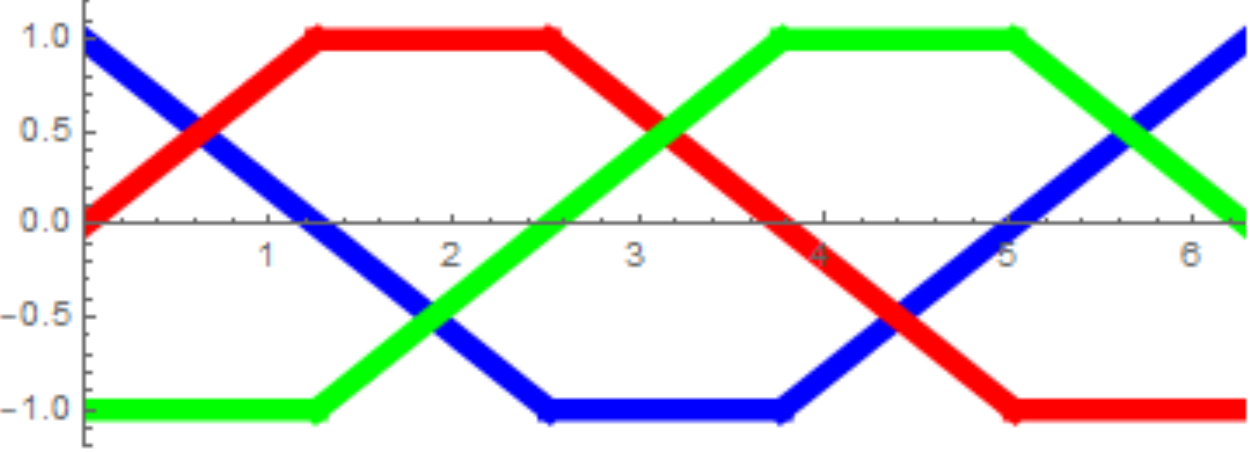}\quad
\labellist
\large
\pinlabel c) at 0 130
\pinlabel $L_{C_1,1}$ at 390 120
\pinlabel $L_{C_2,2}$ at 390 60
\pinlabel $L_{C_2,1}$ at 390 5
\endlabellist
\includegraphics[height=2cm]{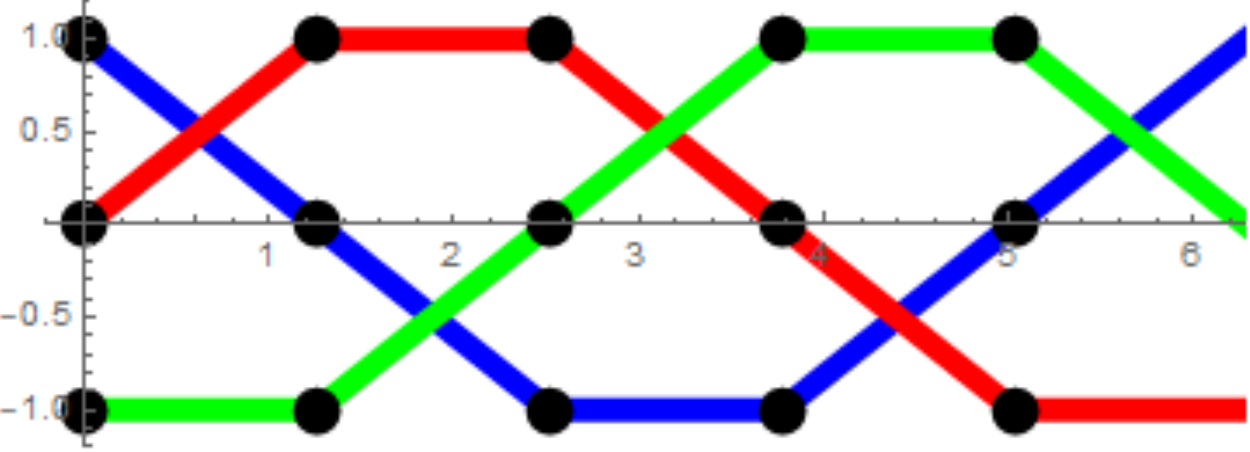}\\
\vspace{0.5cm}
\labellist
\large
\pinlabel d) at 0 120
\pinlabel $D_{C_1}$ at 319 120
\endlabellist
\includegraphics[height=2cm]{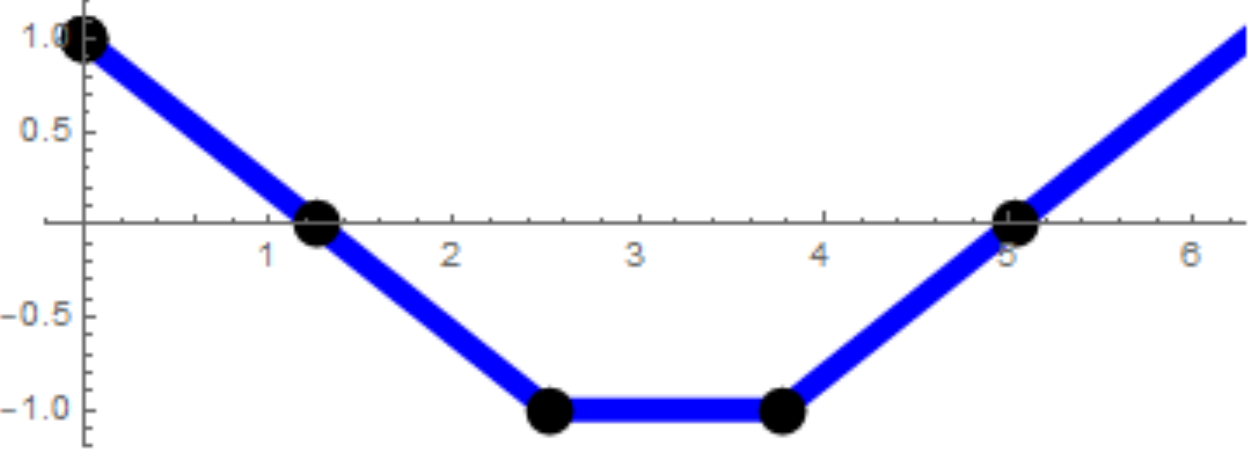}\quad
\labellist
\large
\pinlabel e) at 0 120
\pinlabel $D_{C_2}$ at 370 90
\endlabellist
\includegraphics[height=2cm]{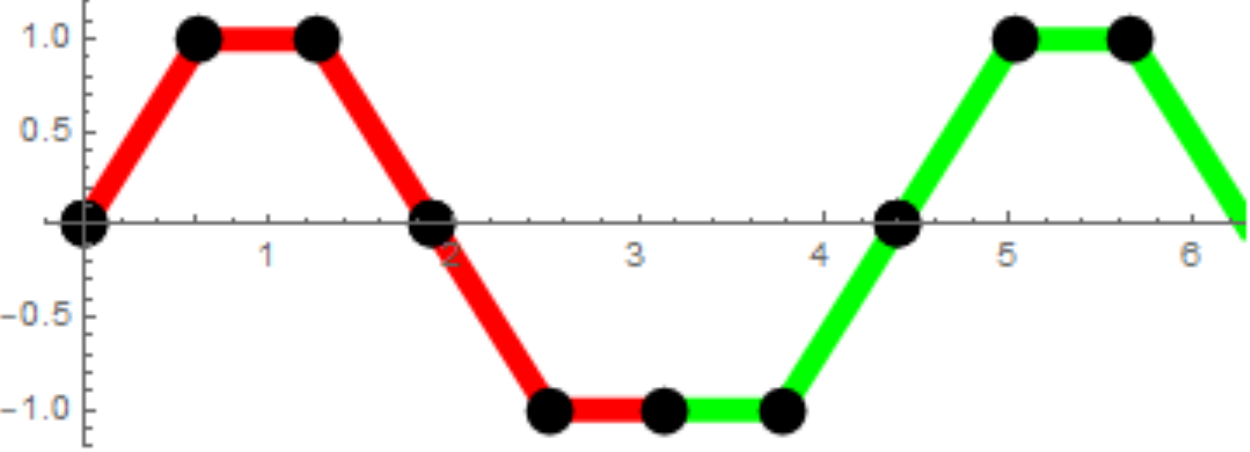}\\
\vspace{0.5cm}
\labellist
\large
\pinlabel f) at 0 230
\pinlabel $F_{C_1}$ at 330 182
\endlabellist
\includegraphics[height=3.4cm]{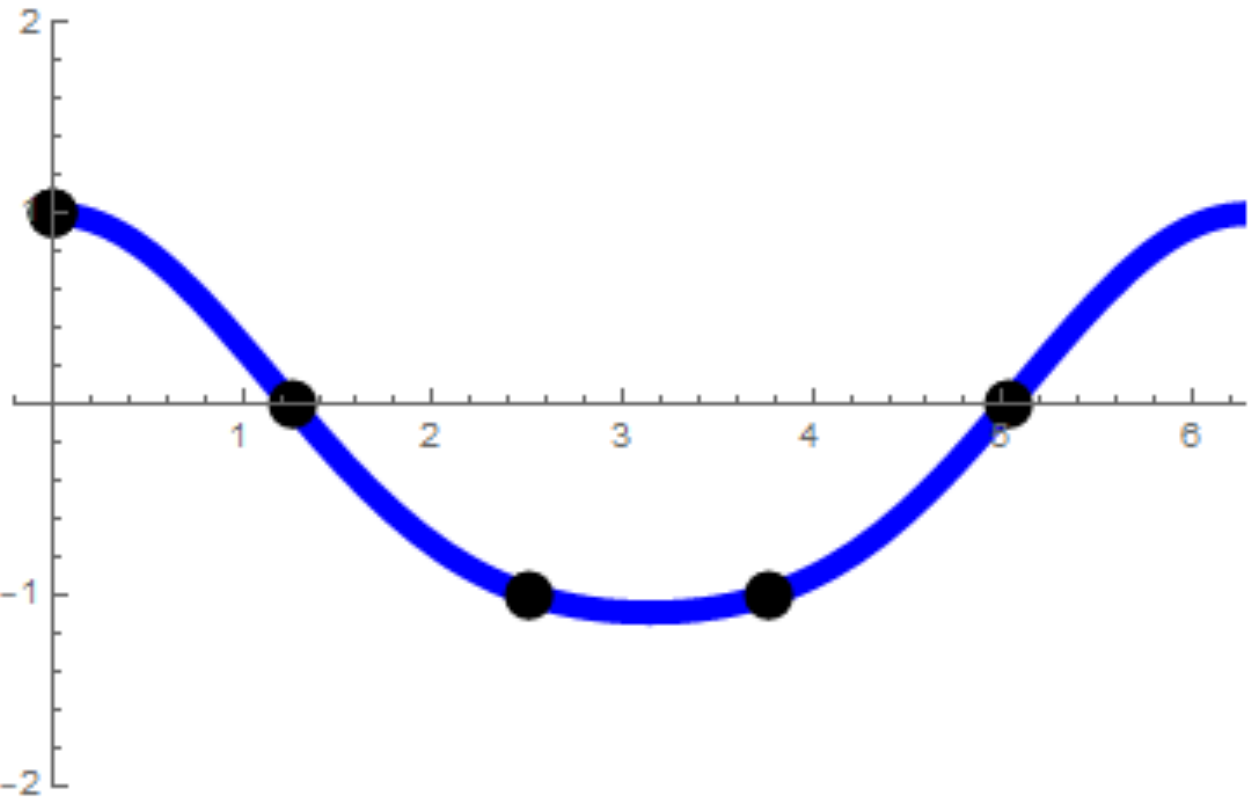}\quad\quad
\labellist
\large
\pinlabel g) at 0 230
\pinlabel $F_{C_2}$ at 370 140
\endlabellist
\includegraphics[height=3.4cm]{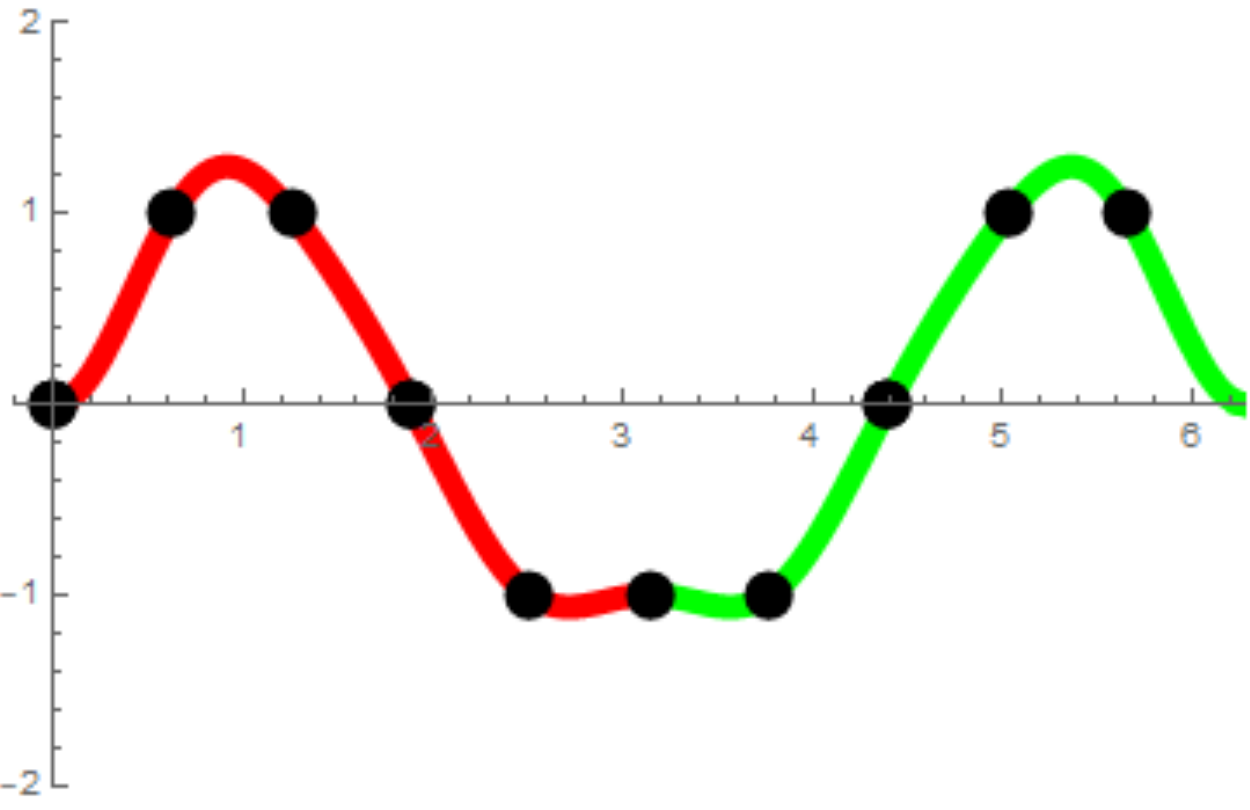}\\
\vspace{0.5cm}
\labellist
\large
\pinlabel h) at 0 230
\endlabellist
\includegraphics[height=3.5cm]{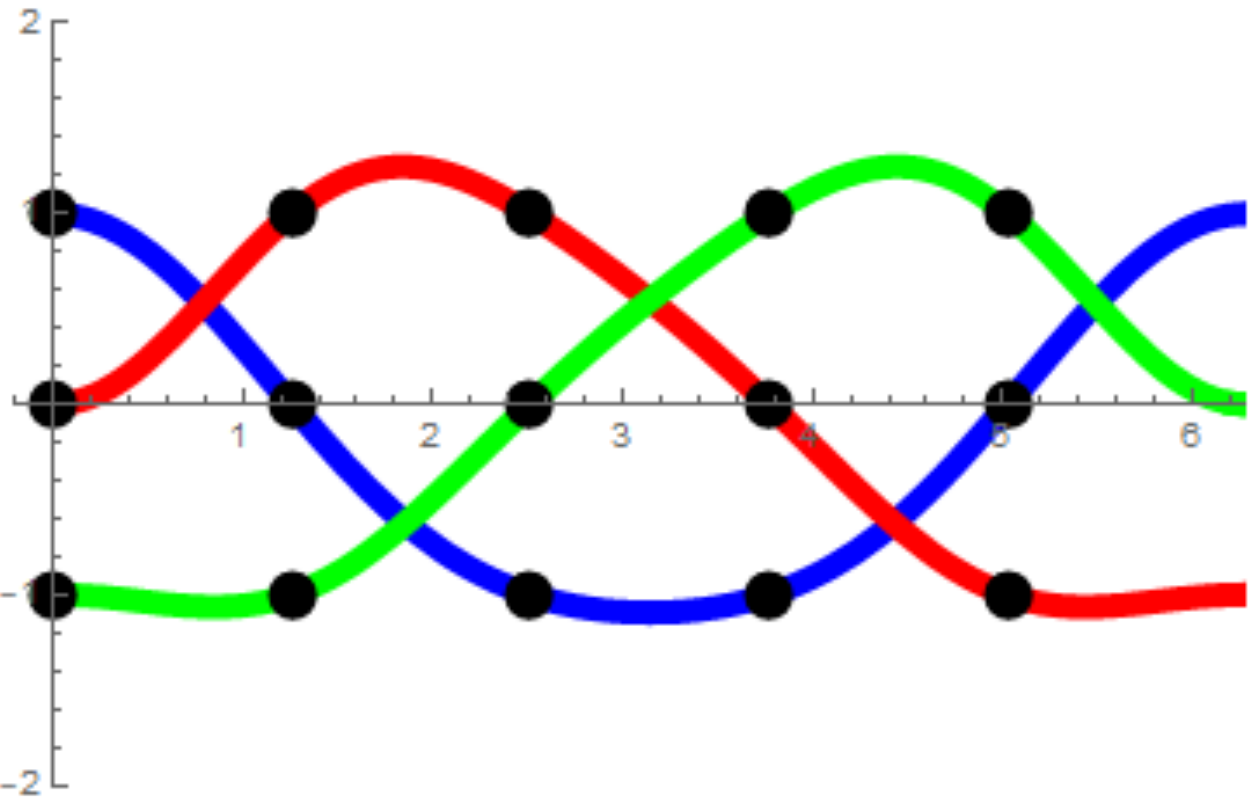}
\caption{Trigonometric interpolation for $F_{C_1}$ and $F_{C_2}$ for a braid that closes to the Whitehead link. 
\label{fig:whitehead}}
\end{figure}
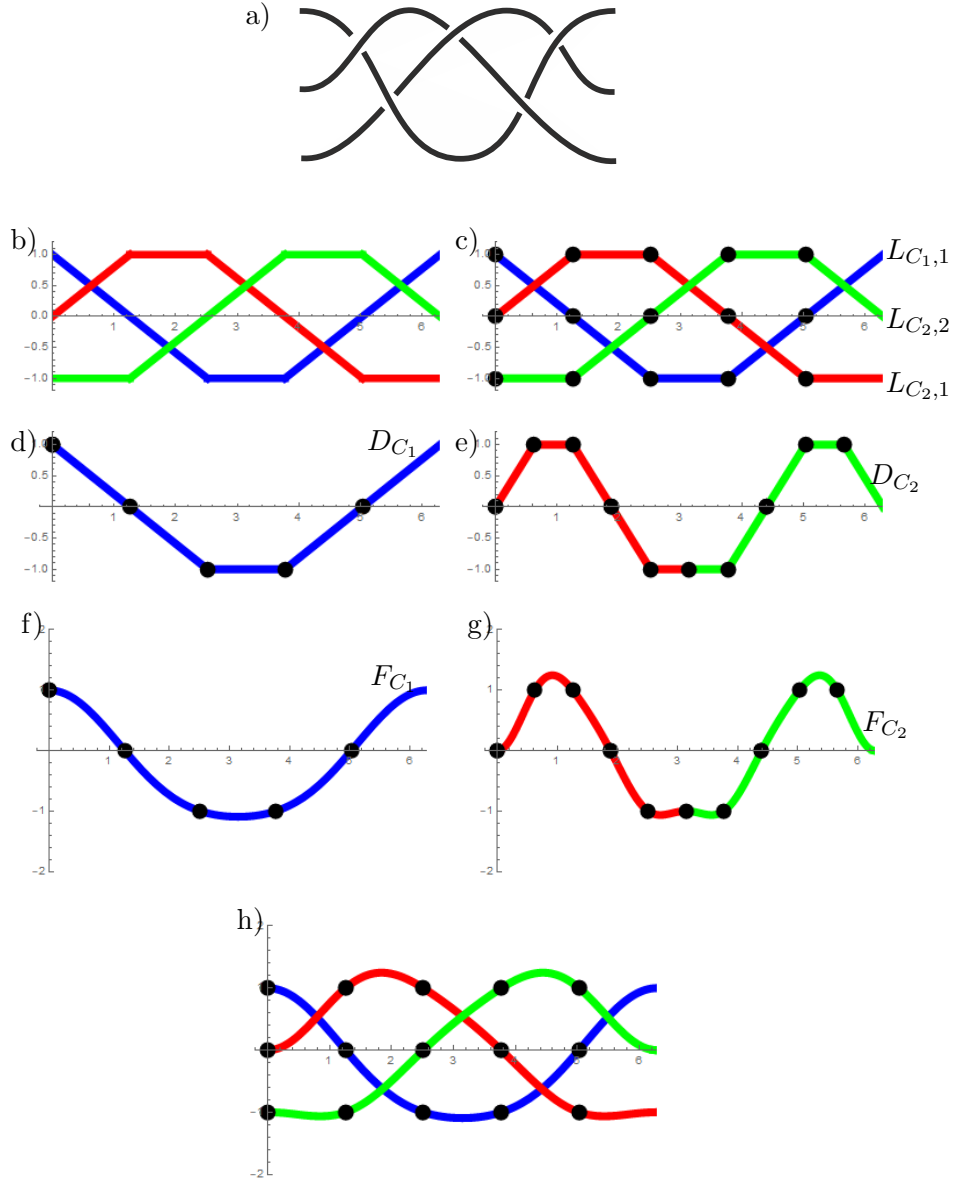

Suppose the braid that we want to construct is $B=\sigma_1^{-1}\sigma_2\sigma_1^{-1}\sigma_2\sigma_1^{-1}$ shown in Figure \ref{fig:whitehead}a), which closes to the Whitehead link. We call its two components $C_1$ and $C_2$. The numbers $s_1$ and $s_2$ denote the numbers of strands that make up $C_1$ and $C_2$, respectively. Hence $s_1\defeq s_{C_1}=1$ and $s_2\defeq s_{C_2}=2$, while the total number of strands is $s=s_1+s_2=3$. The length of the braid is the number of crossings $\ell=5$. We consider a braid diagram of the braid that we want to construct, but neglect the signs of the crossings for now. This results in a braid diagram as in Figure \ref{fig:whitehead}b). We can interpret this braid diagram without signs as the union of the graphs of some piecewise linear functions $L_{C_i,j}:[0,2\pi]\to\mathbb{R}, t\mapsto L_{C_i,j}(t)$, $i=1,2$, $j=1,\ldots,s_i$, such that $L_{C_1,1}(2\pi)=L_{C_1,1}(0)$, $L_{C_2,1}(2\pi)=L_{C_2,2}(0)$ and $L_{C_2,2}(2\pi)=L_{C_2,1}(0)$. Furthermore, we can require that the crossings occur at $t_k=\tfrac{2\pi k}{5}+\tfrac{\pi}{5}$, $k=0,1,2,3,4$.

From this we obtain one piecewise linear function $D_{C_i}:[0,2\pi]\to\mathbb{R}$ for every component $C_i$ via 
\begin{equation}
D_{C_i}\left(\frac{t+2\pi (j-1)}{s_i}\right)=L_{C_i,j}(t), \quad t\in[0,2\pi], \quad i=1,2,\quad j=1,\ldots,s_i. 
\end{equation}

The graphs of $D_{C_i}$, $i=1,2$, are shown in Figure \ref{fig:whitehead}d) and e), respectively.

We now perform two trigonometric interpolations, one for each component. For each $i=1,2$ we interpolate the data points $\left(\tfrac{t_k-\tfrac{\pi}{5}}{s_i}+\tfrac{2\pi (j-1)}{s_i},L_{C_i,j}\left(t_k-\tfrac{\pi}{5}\right)\right)$ and call the resulting trigonometric polynomial $F_{C_i}$. The functions $F_{C_i}\left(\tfrac{t+2\pi (j-1)}{s_i}\right)$, $i=1,2$, $j=1,\ldots,s_i$, are the parametrisations of the $x$-coordinates of the strands in the component $C_i$. The interpolating functions are
\begin{equation}
F_{C_1}(t)=-0.200+1.047\cos(t)+0.153\cos(2t),
\end{equation}
shown in Figure \ref{fig:whitehead}f) and 
\begin{align}
F_{C_2}(t)=&0.100+0.971\cos(t)-0.524\cos(2t)\nonumber\\
&-0.371\cos(3t)-0.076\cos(4t)-0.100\cos(5t),
\end{align}
shown in Figure \ref{fig:whitehead}g). All coefficients are rounded to 3 significant digits.

\begin{remark}\label{rem:case}
In the case of our example, the graphs of the functions $F_{C_i}\left(\tfrac{t+2\pi (j-1)}{s_i}\right):[0,2\pi]\to\mathbb{R}$ form the same crossing pattern as the desired braid, see Figure \ref{fig:whitehead}h). 
However, it is important to note that in general the union of the graphs of the functions $F_{C_i}\left(\tfrac{t+2\pi (j-1)}{s_i}\right):[0,2\pi]\to\mathbb{R}$ do not necessarily form the same crossing pattern as the desired braid $B$. Any set of interpolating functions permutes the strands between $\tfrac{2\pi k}{\ell}$ and $\tfrac{2\pi (k+1)}{\ell}$, $k=0,1,2,\ldots,\ell-1$, where $\ell$ is the length of the braid word, i.e., the number of crossings, in exactly the same way as the desired braid. But for any such permutation there are infinitely many possible braids. An example of this is indicated in Figure \ref{fig:case}, where Figure \ref{fig:case}a) shows the desired crossing pattern between $t=0$ and $t=2\pi/5$ and Figure \ref{fig:case}b) shows the graphs of some interpolating functions that induce the same permutation, but do not have the same crossing pattern. Figure \ref{fig:case}c) shows how we can choose the signs of the crossings of the braid in Figure \ref{fig:case}b) such that the resulting braid is isotopic to the one in Figure \ref{fig:case}a) with the desired crossing sign.
\end{remark}

\begin{figure}
\centering
\labellist
\large
\pinlabel a) at 0 250
\endlabellist
\includegraphics[height=2.5cm]{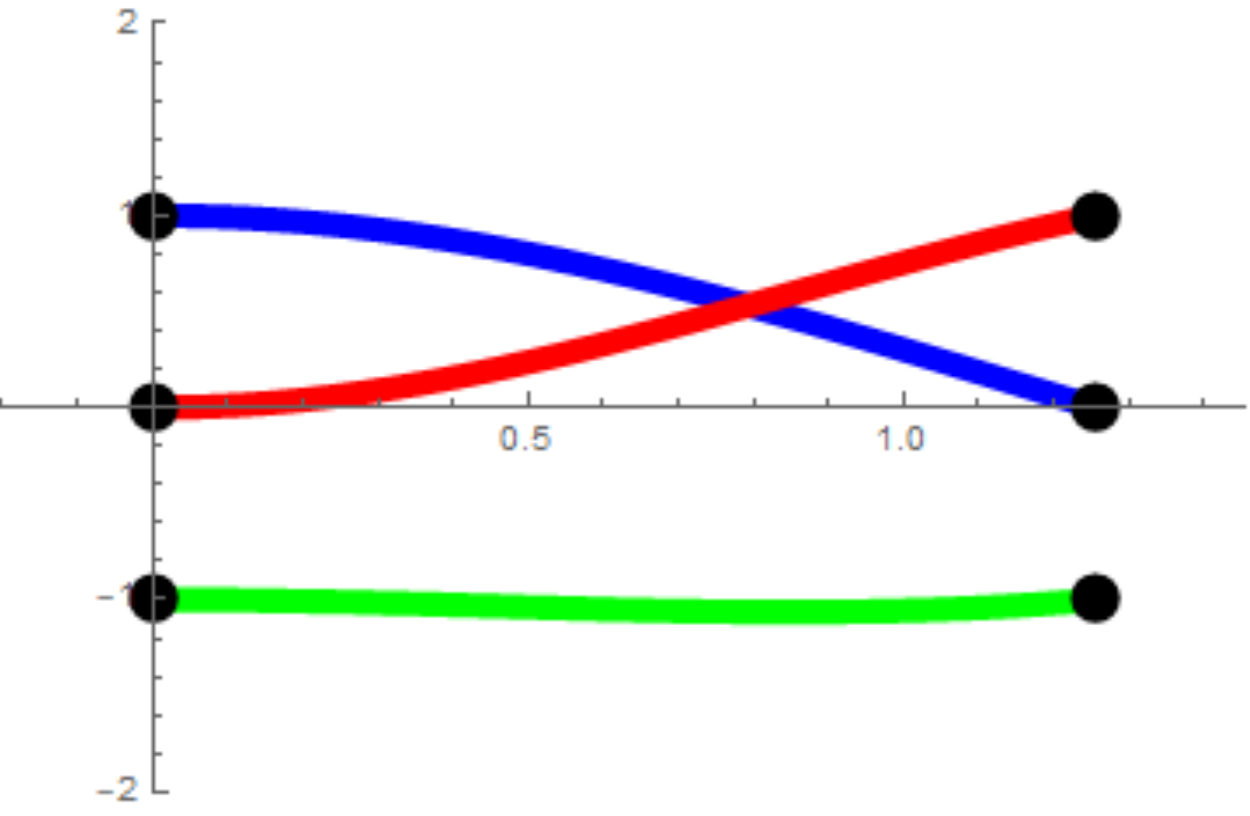}
\labellist
\large
\pinlabel b) at 0 250
\endlabellist
\includegraphics[height=2.5cm]{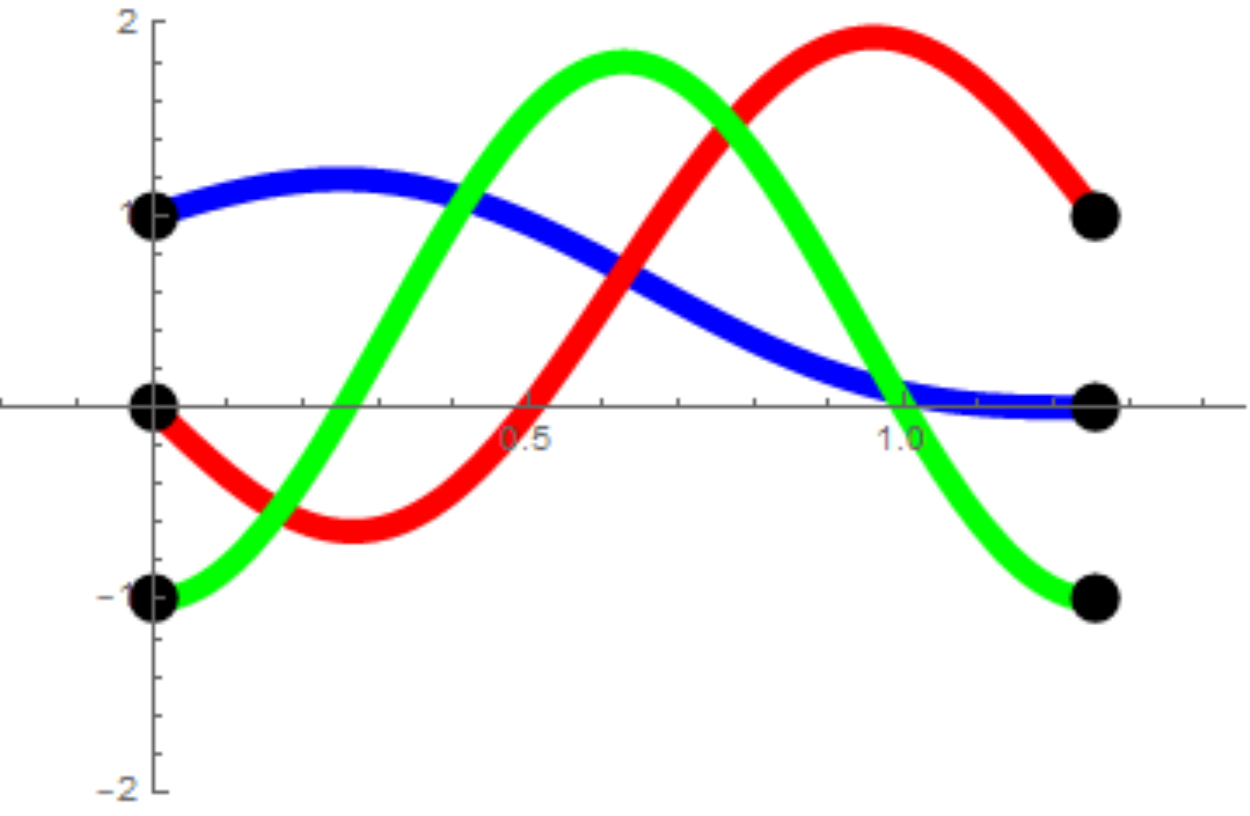}
\labellist
\large
\pinlabel c) at 0 250
\endlabellist
\includegraphics[height=2.5cm]{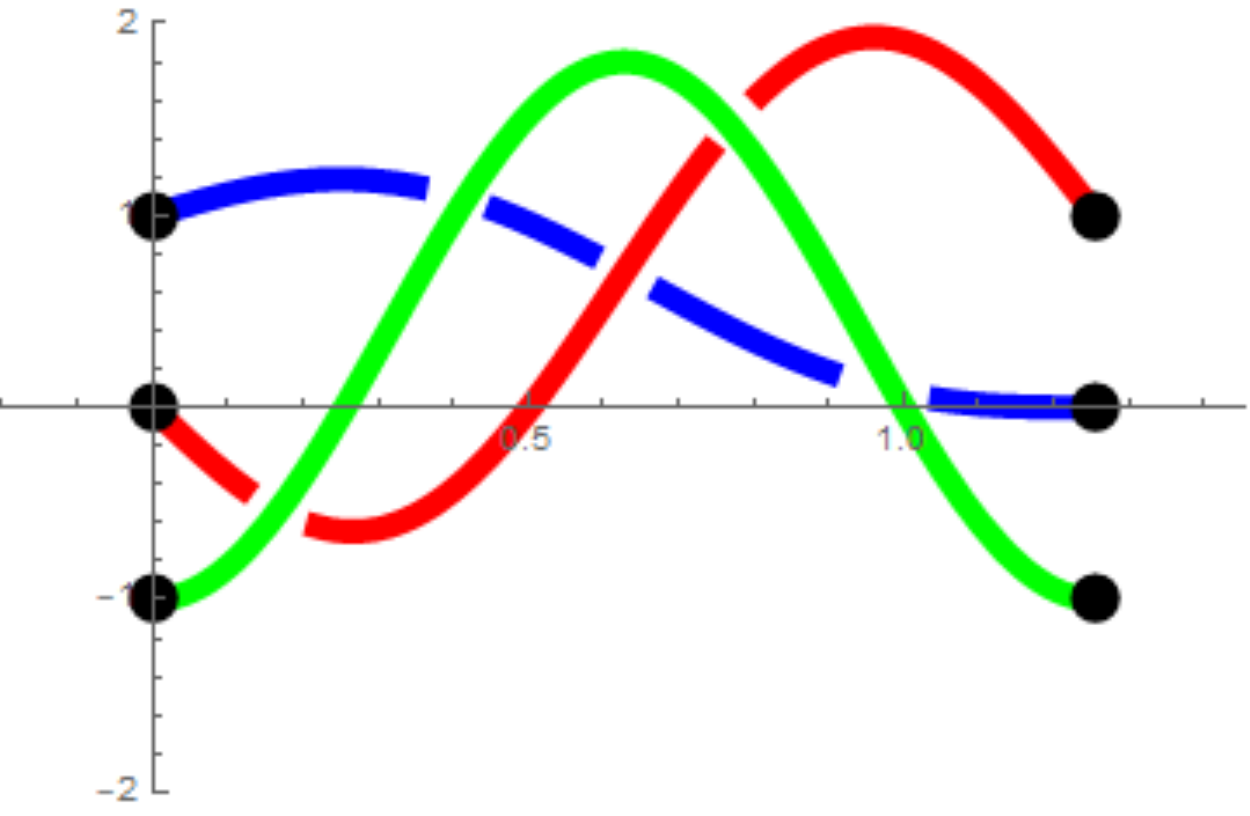}
\caption{An example of graphs of interpolating functions.  
\label{fig:case}}
\end{figure}

In any case, we can perform a trigonometric interpolation for the functions $G_{C_i}$ such that the resulting trigonometric polynomials provide us with a parametrisation of a braid that is isotopic to $B$ as follows. We interpret the union of the graphs of $t\mapsto F_{C_i}\left(\tfrac{t+2\pi (j-1)}{s_i}\right)$, $[0,2\pi]\to\mathbb{R}$, $i=1,2$, $j=1,\ldots,s_i$, as a braid diagram without signs of crossings. If a crossing occurs between the strands $(C_i,j)$ and $(C_{i'},j')$ at $t=t_{*}$, i.e., if $F_{C_i}\left(\tfrac{t_{*}+2\pi (j-1)}{s_i}\right)=F_{C_{i'}}\left(\tfrac{t_{*}+2\pi (j'-1)}{s_{i'}}\right)$, then we add the two points $\left(\tfrac{t_{*}+2\pi (j-1)}{s_i},y\right)$ and $\left(\tfrac{t_{*}+2\pi (j'-1)}{s_{i'}},y'\right)$ to our list of data points for $C_i$ and $C_{i'}$, respectively. If the desired braid $B$ has a crossing between the same pair of strands in the same interval of length $\tfrac{2\pi}{\ell}$, we can choose $y$ and $y'$ such that the crossing obtains the required sign, i.e., $y>y'$ if $(C_i,j)$ is passing over $(C_{i'},j')$ and vice versa. If the crossing is one of the additional, unwanted crossings as in Figure \ref{fig:case}b), we can choose the sign of the crossing, i.e., $y>y'$ or $y<y'$, such that the unwanted crossings cancel in pairs (cf. Figure \ref{fig:case}c)) and the resulting braid is isotopic to the desired one. This is for example achieved by assigning to each strand $(C_i,j)$ for each interval $[\tfrac{2\pi k}{\ell},\tfrac{2\pi(k+1)}{\ell})$ a real number $y_{i,j,k}$ such that they are pairwise distinct and $y_{i,j,k}>y_{i',j',k}$ if in the unique crossing of the desired braid $B$ in the interval $[\tfrac{2\pi k}{\ell},\tfrac{2\pi(k+1)}{\ell})$ the strand $(C_i,j)$ passes over the strand $C_{i',j'}$. The number $y_{i,j,k}$ is the chosen value for all data points for the strand $(C_i,j)$ coming from crossings in the intverval $[\tfrac{2\pi k}{\ell},\tfrac{2\pi (k+1)}{\ell})$.

We call the interpolating trigonometric polynomial $G_{C_i}$. In the case of our example we obtain the following data points for $C_1$ and $C_2$:
\begin{align}
\label{eq:classdata}
C_1:\quad\{&(0.794,-1),(1.857,1),(4.426,-1),(5.490,1)\},\\
C_2:\quad\{&(0.794/2,1),(3.142/2,-1),(4.426/2,1),(1.857/2+\pi,-1),\nonumber\\
&(5.490/2+\pi,-1),(3.142/2+\pi,1)\},\nonumber
\end{align}
where again numbers are rounded. We have chosen the $y$-coordinate of each data point to be 1 or $-1$ depending on whether the corresponding strand is over- or underpassing.

The interpolating functions are given by
\begin{align}
G_{C_1}(t)=&0.520-0.721\cos(t)-0.341\sin(t)+0.860\cos(2t)-1.243\sin(2t),\nonumber\\
G_{C_2}(t)=&-1.630\times10^{8}+1.992\times10^{8}\cos(t)-1.630\times10^{8}\cos(2t)+2.521\times10^8\cos(3t)\nonumber\\
&-19.009\sin(t)+35.090\sin(2t)-19.090\sin(3t).
\end{align}
This shows one potential short-coming of the interpolation method. While the degrees can be taken to be relatively small, the coefficients (and hence the values of the functions) can be very large. Note however, that we have a lot of freedom in choosing the data points. Often a certain variation to these choices leads to much simpler expressions. In our example, we can use
\begin{align}
G_{C_1}(t)=&0.520-0.721\cos(t)-0.341\sin(t)+0.860\cos(2t)-1.243\sin(2t),\\
G_{C_2}(t)=&-0.250\sin(t)-0.75\sin(2t).
\end{align}
instead.

By construction
\begin{equation}
\bigcup_{i=1}^2\bigcup_{j=1}^{s_i}\left\{\left(F_{C_i}\left(\tfrac{t+2\pi (j-1)}{s_i}\right),G_{C_i}\left(\tfrac{t+2\pi (j-1)}{s_i}\right),t\right)| t\in[0,2\pi]\right\}
\end{equation}
is a parametrisation of a braid that closes to the Whitehead link and is of the desired form, i.e., given by trigonometric polynomials.

The key points can be summarized as follows. We obtain a Fourier parametrisation of any braid (up to isotopy) by trigonometric interpolation. We need to perform first an interpolation for every component to obtain the parametrisations of the $x$-coordinates of the strands (given by $F_{C_i}\left(\tfrac{t+2\pi (j-1)}{s_i}\right)$). These parametrisations might lead to additional, unwanted crossings. However, we can perform a trigonometric interpolation for every component to obtain the parametrisations of the $y$-coordinates of the strands (given by $G_{C_i}\left(\tfrac{t+2\pi (j-1)}{s_i}\right)$), which determine the signs of the crossings such that all additional, unwanted crossings cancel. Note in particular that the order of the Steps 1 and 2 in Algorithm 0 matters, since the data points for the second interpolation are only known once the interpolation in Step 1 has been completed.

The degree of the constructed polynomial $f:\mathbb{R}^4\to\mathbb{R}^2$ is closely related to the degrees of the trigonometric polynomials $F_{C}$ and $G_C$. Their degrees are determined by the number of data points used for the interpolation, which in turn are linked to the number of strands and the length of the braid word. 

Trigonometric interpolation is not the only way to obtain trigonometric braid parametrisation. We could instead use a result by Beardon, Carne and Ng \cite{critical} that allows us to lift parametrisations of critical values of polynomials to parametrisations of roots of the polynomials. Since trigonometric polynomials are $C^1$-dense in the set of $2\pi$-periodic $C^1$-functions, we can then approximate the roots by trigonometric polynomials. While trigonometric interpolation has the advantage of low polynomial degrees, this method allows us a certain control over the number of argument-critical points of the constructed polynomials $g_{\lambda}$ and consequentially of $f_\lambda|_{S^3}$. In particular, we have the following result.
\begin{definition}
Let $B$ a braid on $s$ strands. Then $B$ is called \textit{homogeneous} if it can be represented by a word that for all $i\in\{1,2,\ldots,s-1\}$ contains the generator $\sigma_i$ if and only if it does not contain its inverse $\sigma_i^{-1}$. 
\end{definition}
\begin{theorem}
(cf. \cite{bode:2016polynomial}) For every homogeneous braid $B$ we can construct a semiholomorphic polynomial $f:\mathbb{R}^4\to\mathbb{R}^2$ such that the vanishing set $f^{-1}(0)\cap S^3$ on the unit 3-sphere is ambient isotopic to the closure of $B$, say $L$, and $\arg(f):S^3\backslash L\to S^1$ is a fibration.
\end{theorem}

We will see in Section \ref{sec:2dim} that for certain surface braids we can also construct the desired polynomials in such a way that their argument maps are fibration maps over the circle.

\section{The loop braid group}
\label{sec:loop}

We would like to extend the algorithm from Section \ref{sec:review} to knotted surfaces. Since it is based on braid parametrisations, it is natural to consider 2-dimensional generalisations of the Artin braid group. The braid group is the motion group of $n$ unmarked points in the plane $\mathbb{R}^2$. A natural higher-dimensional analogue is therefore the loop braid group.
\begin{definition}
\label{def:loop}
The loop braid group $LB_n$ is the fundamental group of the configuration space of $n$ disjoint unordered unlinked Euclidean circles in the 3-ball $B^3$ under the condition that each of them lies in a plane that is parallel to a given fixed plane.
\end{definition}

This is only one of several possible equivalent definitions. The survey \cite{celeste} offers a good overview of the different interpretations of the loop braid group and related groups. The interpretation of moving circles, which is suggested by Definition \ref{def:loop} is convenient for applications in physics, since it suggests to think of the fourth dimension, the one along the loop in the configuration space, as time.
The loop braid group in Definition \ref{def:loop} is also referred to as the \textit{untwisted ring group}.

The group $LB_n$ has two generators for each $i=1,2,\ldots,n-1$, illustrated in Figure \ref{fig:loopgen}: $\rho_i$, which permutes the $i$th and $i+1$st ring by moving the rings around each other, and $\sigma_i$, which permutes $i$th and $i+1$st ring by shrinking the $i$th ring and threading it through the $i+1$st ring as shown in Figure \ref{fig:loopgen}. More precisely, if the $i$th and $i+1$st rings are initially placed in the $xy$-plane in $\mathbb{R}^3$ with coordinates $x$, $y$ and $z$, then $\sigma_i$ first shrinks the $i$th ring and moves it to the right with $z>0$, then moves it through the hole of the $i+1$st ring by decreasing its $z$-coordinate. Then the $i+1$st ring moves to the left in the $xy$-plane, while the $i$th ring is raised into the $xy$-plane so that after increasing its radius the $i$th and $i+1$st rings have exchanged their positions. For $\rho_i$, the $i$th and $i+1$st rings change positions while staying in the $xy$-plane. The group relations are as follows:
\begin{align}
\sigma_i\sigma_j&=\sigma_j\sigma_i,&\text{ for }|i-j|>1,\label{eq:relll1}\\
\sigma_i\sigma_{i+1}\sigma_i&=\sigma_{i+1}\sigma_i\sigma_{i+1},&\text{ for }i=1,2,\ldots,n-2,\\
\rho_i\rho_j&=\rho_j\rho_i,&\text{ for }|i-j|>1,\\
\rho_i\rho_{i+1}\rho_i&=\rho_{i+1}\rho_i\rho_{i+1}&\text{ for }i=1,2,\ldots,n-2,\\
\rho_i^2&=e,&\text{ for }i=1,2,\ldots,n-1,\\
\rho_i\sigma_j&=\sigma_j\rho_i&\text{ for }|i-j|>1,\\
\rho_{i+1}\rho_i\sigma_{i+1}&=\sigma_i\rho_{i+1}\rho_i&\text{ for }i=1,2,\ldots,n-2,\\
\sigma_{i+1}\sigma_i\rho_{i+1}&=\rho_i\sigma_{i+1}\sigma_i&\text{ for }i=1,2,\ldots,n-2.\label{eq:relll2}
\end{align}

Like for classical braids every loop braid has a unique closure (up to isotopy), which is obtained by identifying the rings at $t=0$ with the rings at $t=2\pi$. Note that this results in linked or knotted tori in $\mathbb{R}^3\times S^1\subset\mathbb{R}^4\subset S^4$.

The elements of the Artin braid group can be interpreted as geometric braids, i.e., $n$ strings in $\mathbb{R}^2\times[0,2\pi]$ which intersect each horizontal plane $\mathbb{R}^2\times\{t\}$ in exactly $n$ points. We can weaken this definition to obtain the \textit{singular braid monoid}. For a \textit{singular braid} we allow finitely many values of $t$, for which the strands intersect $\mathbb{R}^2\times\{t\}$ in $n-1$ points. Two singular braids are equivalent if they are isotopic through a sequence of singular braids, while the start and end points of the strands in $\mathbb{R}^2\times\{0\}$ and $\mathbb{R}^2\times\{2\pi\}$ are fixed. The equivalence classes of singular braids on $n$ strands form a monoid, which is generated (as a monoid) by $\sigma_i$, $\rho_i$ and $\rho_i^{-1}$, $i=1,2,\ldots,n-1$, and subject to the relations:
\begin{align}
\rho_i\rho_i^{-1}=\rho_i^{-1}\rho_i&=e, &\text{ for }i=1,2,\ldots,n-1,\label{eq:rel1}\\
\sigma_i\sigma_j&=\sigma_j\sigma_i,&\text{ for }|i-j|>1,\\
\rho_i\rho_j&=\rho_j\rho_i,&\text{ for }|i-j|>1,\\
\rho_i\rho_{i+1}\rho_i&=\rho_{i+1}\rho_i\rho_{i+1}&\text{ for }i=1,2,\ldots,n-2,\\
\rho_i^{\pm1}\sigma_j&=\sigma_j\rho_i^{\pm1}&\text{ for }|i-j|>1,\\
\rho_{i+1}\rho_i\sigma_{i+1}&=\sigma_i\rho_{i+1}\rho_i&\text{ for }i=1,2,\ldots,n-2,\\
\rho_i\rho_{i+1}\sigma_i&=\sigma_{i+1}\rho_i\rho_{i+1},&\text{ for }i=1,2,\ldots,n-2,\label{eq:rel2}\\
\sigma_i\rho_i^{\pm1}&=\rho_i^{\pm1}\sigma_i&\text{ for }i=1,2,\ldots,n-1\label{eq:rel3}.
\end{align} 
The diagrammatic interpretation of an element of the singular braid monoid differs from classical braids only in that there is a new type of crossing, corresponding to the generator $\sigma_i$, which signifies a transverse intersection of two neighbouring strands in a diagram and does not have an inverse. Such a crossing is called a \textit{singular crossing} or \textit{welded crossing}. (Note that our usage of the symbols $\sigma_i$ and $\rho_i$ is different from that in the literature, where they are actually swapped. See Remark 3.2 below.) Diagrams for the generators are shown in Figure \ref{fig:loopgen}c). Singular braids were introduced by Baez \cite{baez} and Birman \cite{birman}. A similar notion called \textit{welded braids} was introduced by Fenn, Rimányi and Rourke in \cite{fenn}. It is shown that the welded braid group has a group presentation with generators $\sigma_i$ and $\rho_i$, $i=1,2,\ldots,n-1$ and relations (\ref{eq:relll1}) through (\ref{eq:relll2}), that is the welded braid group is isomorphic to the loop braid group.

We call a singular braid a \textit{signed singular braid} if all intersection points, i.e., all `welded crossings' $\sigma_i$, are labelled with a plus- or a minus-sign. A signed singular braid can be represented by a word in $\sigma_i(+)$, $\sigma_i(-)$, $\rho_i$ and $\rho_i^{-1}$, where $\sigma_i(+)$ denotes a singular crossing labeled with a plus-sign and $\sigma_i(-)$ denotes a singular crossing with a minus-sign as in Figure \ref{fig:loopgen}c).

Thus a word in $\sigma_i(+)$, $\sigma_i(-)$, $\rho_i$ and $\rho_i^{-1}$, $i=1,2,\ldots,n-1$, has two interpretations: One as a loop braid, where the generator $\sigma_i$ corresponds to $\sigma_i(+)$ and $\sigma_i^{-1}$ to $\sigma_i(-)$, and the other as a signed singular braid, where all intersection points of strands are marked either by a plus- or a minus-sign. \textit{The relation between these two interpretations is that the signed singular braid can be thought of as the core, or the centre line, of the loops in the loop braid.} Figure \ref{fig:loopgen} illustrates this, using broken surface diagrams and diagrams of signed singular braids. Figure \ref{fig:loopgen}a) displays the generators $\sigma_i$ and $\rho_i$, and the inverse $\sigma_i^{-1}$ as motions of rings in $\mathbb{R}^3$. Figures \ref{fig:loopgen}b) and \ref{fig:loopgen}c) show the broken surface diagrams and the signed singular braid diagrams corresponding to each of these motions, respectively. 
 
\begin{remark}
It should be noted that this identification of (signed) singular braids and loop braids is not the standard in the literature (cf. \cite{celeste, fenn}). Usually the correspondence of the two types of generators is exactly swapped, i.e., intersection points of singular braids correspond to the loop braid generator $\rho_i$ and classical crossings to the loop braid generator $\sigma_i$ (cf. for example \cite{celeste}). We will see that for our purposes the interpretation of the centre lines of the loops as singular braids, which requires this deviation from the conventions in the literature, is more appropriate.
\end{remark}

\begin{figure}
\centering
\labellist
\large
\pinlabel a) at 0 700
\pinlabel $x$ at 500 280
\pinlabel $y$ at 240 650
\pinlabel $z$ at 200 150
\pinlabel $\sigma_i$ at 1100 650
\pinlabel $\sigma_i^{-1}$ at 2200 700 
\pinlabel $\rho_i$ at 3250 690
\endlabellist
\includegraphics[height=3cm]{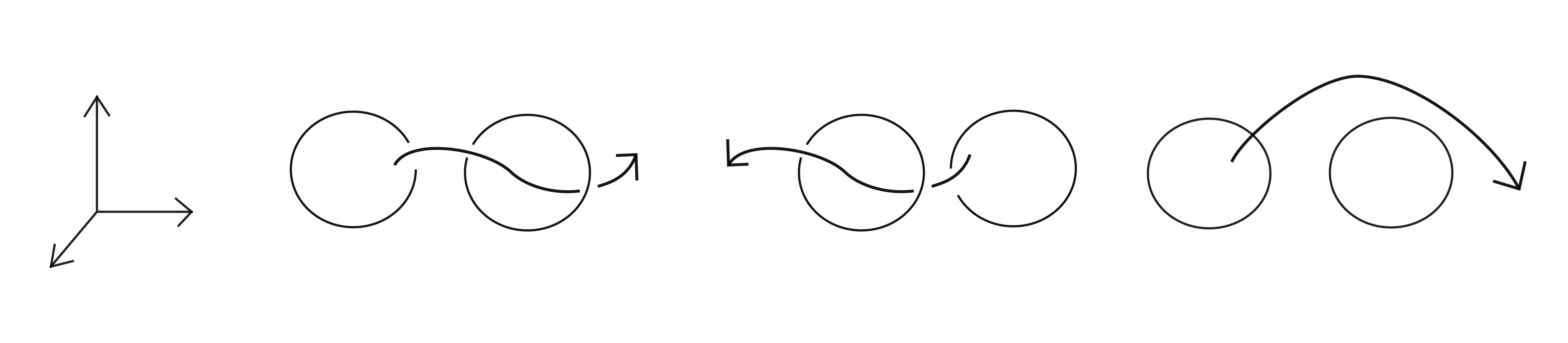}
\labellist
\large
\pinlabel b) at 0 1000
\pinlabel $x$ at 600 280
\pinlabel $t$ at 250 650
\pinlabel $y$ at 190 150
\pinlabel $\sigma_i$ at 1250 1100
\pinlabel $\sigma_i^{-1}$ at 2500 1100
\pinlabel $\rho_i$ at 3700 1100
\endlabellist
\includegraphics[height=3.4cm]{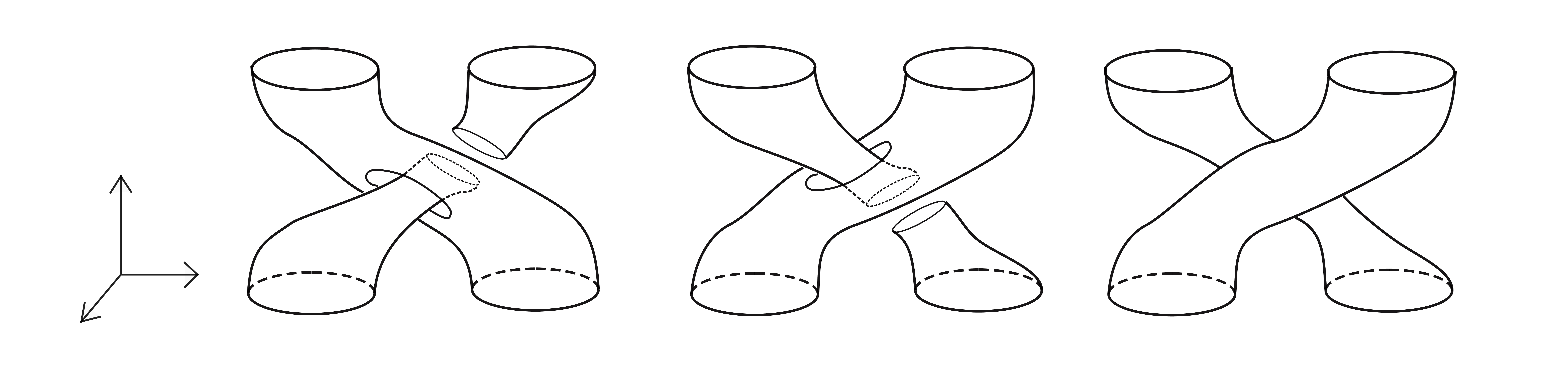}
\labellist
\large
\pinlabel c) at 0 920
\pinlabel $x$ at 400 170
\pinlabel $t$ at 150 600
\pinlabel $\sigma_i(+)$ at 1120 900
\pinlabel $\sigma_i(-)$ at 2200 900
\pinlabel $\rho_i$ at 3400 900
\endlabellist
\includegraphics[height=3.5cm]{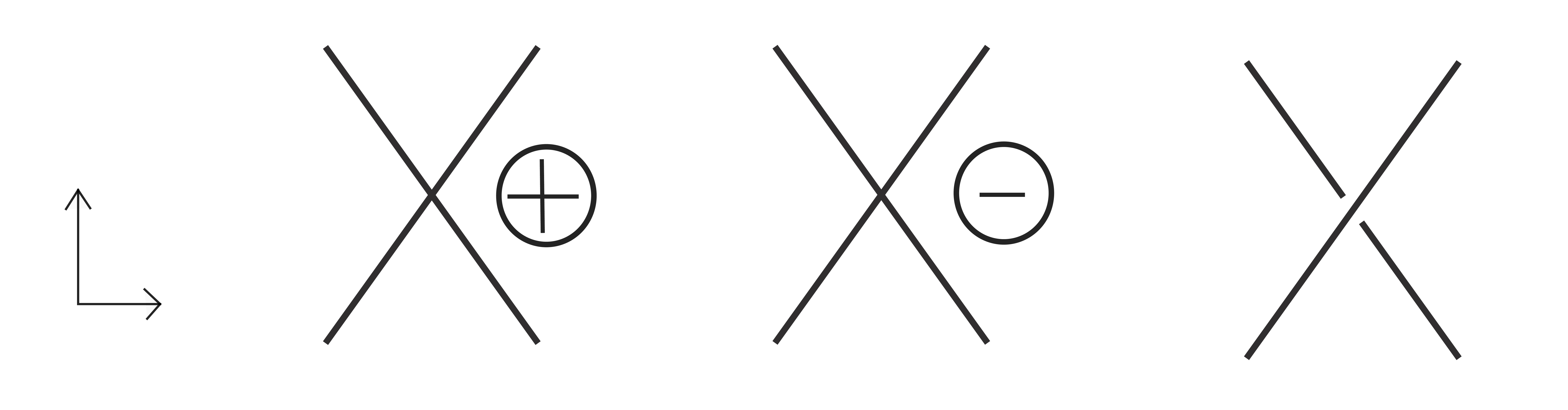}
\caption{The generators of the loop braid group and the singular braid monoid.
\label{fig:loopgen}}
\end{figure}
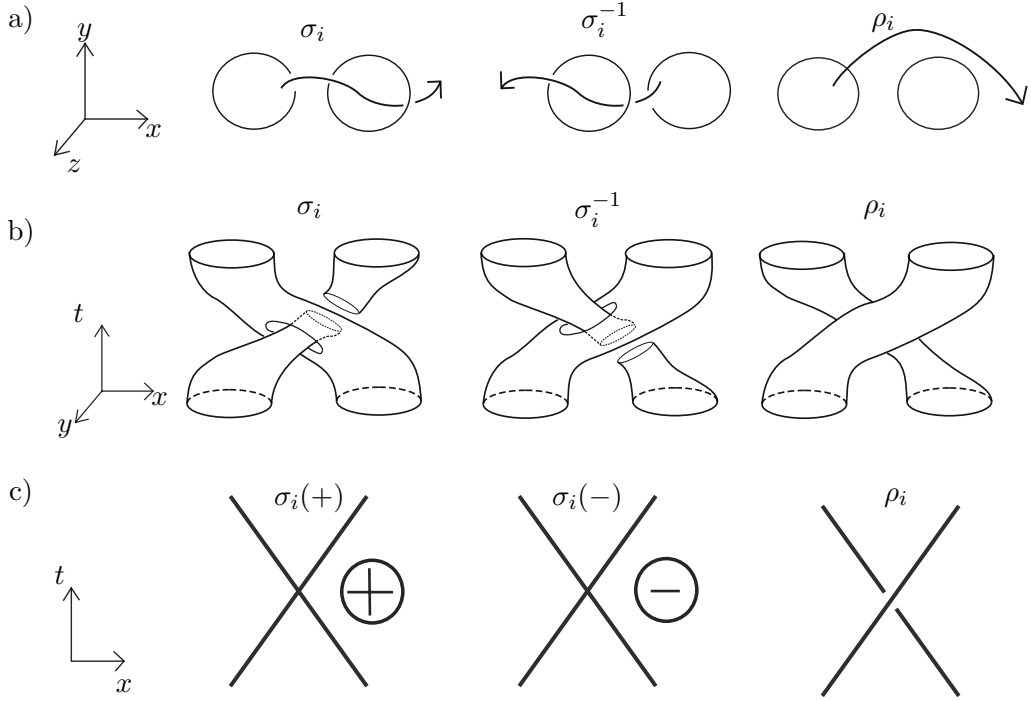

We regard two signed singular braids as equivalent if the underlying singular braids are isotopic, i.e., equivalent modulo relations (\ref{eq:rel1}) through (\ref{eq:rel3}), with matching signs on the singular crossings.

\begin{remark}
\label{rem:equivloopsing}
Note that the equations (\ref{eq:rel1}) through (\ref{eq:rel2}) hold both in the loop braid group and in the singular braid monoid. Furthermore, these equations remain true in the loop braid group when the $\sigma_i$'s are replaced by their inverses. It follows that if two signed singular braids are equivalent modulo Eqs. (\ref{eq:rel1}) through (\ref{eq:rel2}), then the corresponding loop braids are isotopic.
\end{remark}

We proceed by outlining the algorithm that constructs any given loop braid as the vanishing set of a polynomial, which proves Theorem \ref{thm:intro1}. This is followed by an example, a generalisation to some elements of other motion groups and a discussion of an application to the construction of time-dependent vector fields with knotted flow lines.

\subsection{The proof of Theorem \ref{thm:intro1}}
Let $B$ be a loop braid. Without loss of generality we can assume that the given fixed plane in Definition \ref{def:loop} is the $xy$-plane. Like in Section \ref{sec:review} we label the strands $(C,j)$, where $C$ is running through the set of components of the closure of the loop braid and $j$ enumerates the strands in the component $C$, i.e., $j=1,2,\ldots,s_C$, where $s_C$ is the number of strands in $C$. Now suppose that for each strand there are functions $x_{C,j}$, $y_{C,j}$, $z_{C,j}:[0,2\pi]\to\mathbb{R}$ and $\rho_{C,j}:[0,2\pi]\to\mathbb{R}_{>0}$ that parametrise the $(x,y,z)$-coordinate of the centre line of the corresponding ring and the radius $\rho$ of the ring, respectively. That is, at the time $t$ the ring in question is the circle of radius $\rho_{C,j}(t)$, lying in the plane parallel to the $xy$-plane and centred at $(x_{C,j}(t), y_{C,j}(t),z_{C,j}(t))$. Then a point $(x,y,z,t)\in\mathbb{R}^3\times[0,2\pi]$ lies on the strand $(C,j)$ if and only if 
\begin{align}
\label{eq:p}
P_{C,j}(x,y,z,t)&\defeq(x-x_{C,j}(t))^2+(y-y_{C,j}(t))^2-\rho_{C,j}(t)^2+\rmi (z-z_{C,j}(t))\nonumber\\
&=0.
\end{align}

Hence for all values of $\lambda>0$ the loop braid is the vanishing set of the function
$g_{\lambda}:\mathbb{R}^3\times[0,2\pi]\to\mathbb{C}\cong\mathbb{R}^2$,
\begin{equation}
\label{eq:looppoly}
g_{\lambda}(x,y,z,t)=\prod_C \prod_{j=1}^{s_C}P_{C,j}\left(\frac{x}{\lambda},\frac{y}{\lambda},\frac{z}{\lambda},t\right).
\end{equation}

\begin{lemma}
If for every component $C$ there are trigonometric polynomials $F_C$, $G_C$, $H_C:[0,2\pi]\to\mathbb{R}$ and $R_C:[0,2\pi]\to\mathbb{R}_{>0}$ such that 
\begin{align}
(x_{C,j}(t),y_{C,j}(t),z_{C,j}(t),\rho_{C,j}(t))=&\left(F_C\left(\frac{t+2\pi (j-1)}{s_C}\right),G_C\left(\frac{t+2\pi (j-1)}{s_C}\right),\right.\nonumber\\
&\left.H_C\left(\frac{t+2\pi (j-1)}{s_C}\right),R_C\left(\frac{t+2\pi (j-1)}{s_C}\right)\right),
\end{align}
then expanding the product in Eq. (\ref{eq:looppoly}) yields a polynomial in $x$, $y$, $z$, $\rme^{\rmi t}$ and $\rme^{-\rmi t}$. 
\end{lemma}
\begin{proof}
The proof is completely identical to the classical braid case in \cite{bode:2016polynomial} using only induction and basic arithmetic of roots of unity.
\end{proof}


As in \cite{bode:2016polynomial} we replace every instance of $\rme^{\rmi t}$ in the polynomial expression of $g_{\lambda}$ by a complex variable $v$ and $\rme^{-\rmi t}$ by its complex conjugate $\overline{v}$. We call the resulting function $f_{\lambda}:\mathbb{R}^5\cong\mathbb{R}^3\times\mathbb{C}\to\mathbb{C}\cong\mathbb{R}^2$. Note that $f_{\lambda}(x,y,z,\cos(t),\sin(t))=g_{\lambda}(x,y,z,t)$.

It is convenient to consider $f_{\lambda}$ as a family of polynomials in three real variables $x$, $y$ and $z$, parametrised by $v=r\rme^{\rmi t}$. We often write $f_{\lambda}(x,y,z,v)$ instead of $f_{\lambda}(x,y,z,\text{Re}(v),\text{Im}(v))$.

\begin{lemma}
\label{lem:implicit}
There is a $\delta\in(0,1)$ such that for all $r\in[1-\delta,1]$ the vanishing set of $f_{\lambda}(\cdot,\cdot,\cdot,r\rme^{\rmi t})$, $t\in[0,2\pi]$ contains the desired braid $B$, i.e., $B\subset f_\lambda^{-1}(0)\cap(\mathbb{R}^3\times rS^1)$.
\end{lemma}
\begin{proof}
Let $(x,y,z)$ be such that $f_1(x,y,z,\rme^{\rmi t})=0$ for some $t\in[0,2\pi]$. Then it belongs to some strand $(C,j)$. We can turn $f_{1}$ into a function with 3-dimensional image by considering the map 
\begin{align}
\tilde{f}_{C,j}:U\to&\mathbb{R}^2\times S^1;\nonumber\\
(x,y,z,v)\mapsto &\left(\text{Re}(f_{1}(x,y,z,v)),\text{Im}(f_1(x,y,z,v))\right.,\nonumber\\
&\left.\arg\left(x+\rmi y-F_{C}\left(\frac{t+2\pi (j-1)}{s_C}\right)-\rmi G_{C}\left(\frac{t+2\pi (j-1)}{s_C}\right)\right)\right),
\end{align}
where $U$ is an open subset of $\mathbb{R}^3\times\mathbb{C}$ containing $f_1^{-1}(0)\cap (\mathbb{R}^3\times S^1)$, i.e. a neighbourhood of the loop braid. Note that the third component that we added to $f$ can be interpreted as a circular coordinate along the ring that is part of the strand $(C,j)$.

It is easy to see that for all $(C,j)$ the gradient matrix of $\tilde{f}_{C,j}$ with respect to $x$, $y$ and $z$ is invertible for all $(x,y,z,\rme^{\rmi t_*})\in\mathbb{R}^3\times S^1$ with $f_1(x,y,z,\rme^{\rmi t_*})=0$. Furthermore, the function is analytic with respect to all variables. The analytic implicit function theorem then implies that there are neighbourhoods $V_{C,j,\phi,t_*}$ of $(1,\rme^{\rmi t_*})\in\mathbb{R}^{+}\times S^1$ and $W_{C,j,\phi,t_*}$ of $(x,y,z)$ such that for every $\phi\in S^1$ the solutions to $\tilde{f}_{C,j}=(0,0,\phi)$ are real analytic functions of $r$ and $t$. Since the rings in $f_1^{-1}(0)\cap\mathbb{R}^3\times S^1$ are disjoint, we can by continuity assume that they are also disjoint in $\bigcup_{C,j,\phi,t_*}\{(w,r\rme^{\rmi t_*})|w\in W_{C,j,\phi,t_*}\}$ for all $r\in[1-\delta,1]$ for some $\delta\in(0,1)$. Note that varying $r$ between $1-\delta$ and $1$ induces a smooth isotopy between the zeros of $f_{1}(x,y,z,r\rme^{\rmi t})$ in $\bigcup_{C,j,\phi,t_*}\{(w,r\rme^{\rmi t_*})|w\in W_{C,j,\phi,t_*}\}\subset \mathbb{R}^3\times rS^1$ and the zeros of $f_{1}(x,y,z,\rme^{\rmi t})$, since 0 is a regular value throughout.  

The parameter $\lambda$ scales $B$ in the $x$-, $y$- and $z$-direction and in particular does not change isotopy types. Moreover, the value of $\delta$ does not depend on $\lambda$.\end{proof}




\begin{lemma}
\label{lem:smalllambda}
For small enough $\lambda>0$ the closure of $B$ is contained in $f_{\lambda}^{-1}(0)\cap S^4$.
\end{lemma}
\begin{proof}
Let $M(\lambda,r)=\max\{|(x,y,z)|:f_{\lambda}(x,y,z,r\rme^{\rmi t})=0\}$. Then $M(\lambda,r)=\lambda M(1,r)$. Thus if we choose $\lambda$ small enough we have $M(\lambda,1-\delta)<\sqrt{\delta(2-\delta)}$, which is equivalent to the vanishing set $f_\lambda^{-1}(0)\cap(\mathbb{R}^3\times (1-\delta)S^1)$ being contained in the open unit 5-ball. Since the vanishing set of $f_{\lambda}(\cdot,\cdot,\cdot,\rme^{\rmi t})$ is outside the open unit 5-ball, each ring must pass through the unit 4-sphere.

By Lemma \ref{lem:implicit} we can parametrise each component of the vanishing set of $f_{\lambda}$ for $r\in[1-\delta,1]$ by $r$, $t=\arg(v)$ and $\phi$, where $\phi$ corresponds to the variable along a ring. As mentioned before this parametrisation is real analytic, which means that the same arguments as in \cite{bode:2016polynomial} apply. That is for each $(C,j)$, each $t$ and each $\phi$ there is a real analytic function $h_{C,j,t,\phi}:[1-\delta,1]\to\mathbb{R}^5$, whose image consists of the points on the vanishing set of $f_{\lambda}$ that belong to the $j$th strand of the component $C$ with parameters $t$ and $\phi$. Furthermore, for each $C$, $j$, $t$ and $\phi$ this image of $h_{C,j,t,\phi}$ intersects $S^4$ in a unique point if $\lambda$ is small enough. The proof of this uses the same arguments as in \cite{bode:2016polynomial}, i.e., applying Rolle's Theorem to a continuously differentiable function that is constructed from $h_{C,j,t,\phi}$ and whose zeros are intersection points of the image of $h_{C,j,t,\phi}$ and $S^4$. We call this unique intersection point $Z(C,j,t,\phi)$ and the value of $|v|=r$ for which the intersection occurs $z(C,j,t,\phi)$, i.e. $h_{C,j,t,\phi}(z(C,j,t,\phi))=Z(C,j,t,\phi)$.


Let $\Psi:\mathbb{B}^3\times S^1\to S^4\subset\mathbb{R}^3\times\mathbb{C}$ be the projection map 
\begin{equation}
(x,y,z,t)\mapsto(x,y,z,\sqrt{1-x^2-y^2-z^2}\rme^{\rmi t}),
\end{equation}
where $t\in[0,2\pi]/(0\sim 2\pi)\cong S^1$.
We can now define a smooth isotopy from $\Psi(f_{\lambda}^{-1}(0)\cap \mathbb{R}^3\times S^1)$, which we know to be the closure of the desired loop braid $B$, to $\bigcup_{C,j,t,\phi}Z(C,j,t,\phi)$, which proves that 
\begin{equation}
\bigcup_{C,j,t,\phi}Z(C,j,t,\phi)
\end{equation} 
is ambient isotopic to the closure of $B$ in $S^4$ by the isotopy extension theorem. Since by construction $\bigcup_{C,j,t,\phi}Z(C,j,t,\phi)$ is contained in $f_{\lambda}^{-1}(0)\cap S^4$, this finishes the proof of the lemma.

The isotopy is simply 
\begin{equation}
(\Psi(h_{C,j,t,\phi}(1)),s)\mapsto \Psi(h_{C,j,t,\phi}(\max\{1-s,z(C,j,t,\phi)\}))
\end{equation}
for $s\in[0,1]$.
\end{proof}

In theory, the proof provides an outline of how small $\lambda$ has to be chosen in order to guarantee the desired result. We have to calculate the first derivatives of $\tilde{f}_{C,j}$ for all $C$ and $j$ and find the largest value of $r<1$ where $\tilde{f}_{C,j}=(0,\phi)$ and $\det\nabla\tilde{f}_{C,j}=0$ for some $C$, $j$ and $\phi\in S^1$. This leads to the value of $\delta<1-r$. The implicit function theorem gives descriptions of the derivatives of the function $h_{C,j,t,\phi}$. By solving the corresponding system of differential equations we obtain an expression of $h_{C,j,t,\phi}(r)$ for $r\in(1-\delta,1]$, which allows us to find $M=\max\{|(x,y,z)|:h_{C,j,t,\phi}(r)=(x,y,z)\text{ for some }C,j, t, \phi, \text{ and }r\in(1-\delta,1]\}$. This together with knowledge of the derivatives of $h_{C,j,t,\phi}$ is enough to find sufficient values for $\lambda$.

In practice, and in particular for complicated braids, this is computationally too expensive. Instead we usually make an educated guess for $\lambda$, simply plot the vanishing set of $f_{\lambda}^{-1}(0)\cap S^4\cap\{t=\text{constant}\}$ for sufficiently many values of $t$ and check if it leads to the desired loop braid. This method is not perfect, but quite practical.


Compared to the lower-dimensional case in \cite{bode:2016polynomial} the real variables $x$, $y$ and $z$ take the role of the complex variable $u$. As a consequence the polynomials obtained from the outlined procedure lack the semiholomorphicity of the polynomials in \cite{bode:2016polynomial}. While in the lower-dimensional case this guarantees that we always know the number of zeros of the polynomial for a given parameter $v$ (namely simply the degree with respect to the complex variable), we do not have the analogous knowledge about the number of components in this setting. Therefore, the vanishing set of $f_{\lambda}$ on $S^4$ might contain other components that we have no control over. All that we can guarantee is that the closure of $B$ is a subset of the vanishing set on $S^4$.

The previous lemmas show that if we can find a parametrisation of a loop braid $B$ in terms of trigonometric polynomials, then the following algorithm produces polynomials of the desired form.

\noindent
\textbf{Algorithm 1:}
\begin{steps}
\item From the given braid word find the trigonometric polynomials $F_C$.
\item Find the trigonometric polynomials $G_C$. 
\item Find the trigonometric polynomials $H_C$.
\item Find the trigonometric polynomials $R_C$.
\item Define $g_{\lambda}$, write it as a polynomial by expanding the product and define $f_{\lambda}$.
\item Determine how small $\lambda$ has to be chosen.
\end{steps}

Like in the classical case the only computational step that requires a more detailed explanation is the trigonometric interpolation that leads to the functions $F_C$, $G_C$, $H_C$ (the parametrisations of the $x$-, $y$- and $z$-coordinates of the strands) and $R_C$ (the radius of the loops).

\begin{lemma}
We can find a parametrisation of any loop braid $B\in LB_n$ in terms of trigonometric polynomials $F_C$, $G_C$, $H_C$ and $R_C$ via trigonometric interpolation. 
\end{lemma}
\begin{proof}
From a broken surface diagram or equivalently a word in the generators $\sigma_i$ and $\rho_i$, $i=1,2,\ldots n-1$, and their inverses we obtain a diagram of a signed singular braid by using the rule in Figure \ref{fig:loopgen} and using that in the loop braid group $\rho_i^{-1}=\rho_i$. We refer to crossings in this signed singular braid diagram that correspond to generators $\sigma_i$ of the loop braid group (or their inverses) as $\sigma$\textit{-crossings} and to crossings that correspond to generators $\rho_i$ (or their inverses) as $\rho$\textit{-crossings} or \textit{classical crossings}. The term crossing may refer to either a $\sigma$- or a $\rho$-crossing

From the signed singular braid diagram we can extract data points for a trigonometric interpolation for $F_C$ exactly as in Section \ref{sec:review}. This means the values of the data points correspond to the positions of each strand before and after every crossing. The union of the graphs of the functions $F_C\left(\frac{t+2\pi (j-1)}{s_C}\right)$ form a braid diagram $B'$ without specified signs of crossings. Note that like in the classical case the crossing pattern of this braid might be different from that of the desired braid (cf. Remark \ref{rem:case}).

The paragraph after Remark \ref{rem:case} outlines how the data points for the trigonometric interpolation for the $y$-coordinate have to be chosen in the 3-dimensional case, so that the interpolating trigonometric polynomial gives us a parametrisation of a braid that is isotopic to the desired braid. The 4-dimensional case only differs from this in that there is a new type of crossings: the $\sigma$-crossing.

If in $B'$ a crossing occurs between the strands $(C_i,j)$ and $(C_{i'},j')$ at $t=t_*$, i.e., if $F_{C_i}\left(\tfrac{t_*+2\pi (j-1)}{s_i}\right)=F_{C_{i'}}\left(\tfrac{t_*+2\pi (j'-1)}{s_{i'}}\right)$, then we add the two points $\left(\tfrac{t_*+2\pi (j-1)}{s_i},y\right)$ and $\left(\tfrac{t_*+2\pi (j'-1)}{s_{i'}},y'\right)$ to our list of data points for $C_i$ and $C_{i'}$ respectively, where the data values $y$ and $y'$ are determined as follows. We assign to every strand $(C_i,j)$ for each interval $[\tfrac{2\pi k}{\ell},\tfrac{2\pi (k+1)}{\ell})$, $k=0,1,2,\ldots,\ell-1$, a real number $y_{i,j,k}$ such that they are pairwise distinct and $y_{i,j,k}>y_{i',j',k}$ if the crossing of the desired loop braid $B$ in the interval $[\tfrac{2\pi k}{\ell},\tfrac{2\pi (k+1)}{\ell})$ is a $\rho$-crossing, where the strand $(C_i,j)$ is passing over the strand $(C_{i'},j')$. If the crossing of the desired loop braid $B$ in the interval $[\tfrac{2\pi k}{\ell},\tfrac{2\pi (k+1)}{\ell})$ is a $\sigma$-crossing, there should be no $y_{\cdot,\cdot,k}$ between the two values assigned to the two strands that are involved in this crossing.

Then the values $y_{i,j,k}$ are chosen to be the values for all data points for the strand $(C_i,j)$ in the interval $[\tfrac{2\pi k}{\ell},\tfrac{2\pi (k+1)}{\ell})$ except for one case. If the desired braid $B$ has a $\sigma$-crossing in the interval $[\tfrac{2\pi k}{\ell},\tfrac{2\pi (k+1)}{\ell})$ between the strands $(C_i,j)$ and $(C_{i'},j')$, then both of the data points that come from the first crossing of $B'$ between $(C_i,j)$ and $(C_{i'},j')$ in the interval $[\tfrac{2\pi k}{\ell},\tfrac{2\pi (k+1)}{\ell})$ obtain the same value $y$, which should lie between $y_{i,j,k}$ and $y_{i',j',k}$. All data points coming from other crossings in the interval, even if they are crossings of the strands $(C_i,j)$ and $(C_{i'},j')$, get the usual data values $y_{i,j,k}$.

The $z$-coordinate is only important in the neighbourhood of a $\sigma$-crossing, where one loop, say the $j$th strand of the component $C_1$, passes through the hole of the other, say the $j'$th strand of $C_2$. If $H_C$ is a trigonometric polynomial that parametrises the $z$-coordinates of the strands of the component $C$, then the sign of $H'_{C_1}\left(\frac{t+2\pi (j-1)}{s_{C_1}}\right)-H'_{C_2}\left(\frac{t+2\pi (j'-1)}{s_{C_2}}\right)$ at the time $t$ of the crossing determines if the $j$th strand of $C_1$ passes the $j'$th strand of $C_2$ from below (positive sign) or from above (negative sign). The data points for this interpolation must therefore contain information about the values of $H_C$ and its derivative $H'_C$.

For a $\sigma$-crossing between the strands $(C_1,j)$ and $(C_2,j')$ at a time $t$, where the strand $(C_1,j)$ comes from above we need two data points to ensure that $H_{C_1}\left(\frac{t+2\pi (j-1)}{s_{C_1}}\right)=H_{C_2}\left(\frac{t+2\pi (j'-1)}{s_{C_2}}\right)$ and two more to ensure that $H'_{C_1}\left(\frac{t+2\pi (j-1)}{s_{C_1}}\right)<H'_{C_2}\left(\frac{t+2\pi (j'-1)}{s_{C_2}}\right)$. We thus have to find trigonometric polynomials $H_{C_j}$ with specified values $z_i$ of $H_{C_j}$ and values $z'_i$ of the derivative $\tfrac{\partial H_{C_j}}{\partial t}$ at a specified set of data points $t_i$, $i=1,2,\ldots,N$:
\begin{equation}
H_{C_j}(t_i)=z_i,\quad \tfrac{\partial H_{C_j}}{\partial t}(t_i)=z'_i.
\end{equation}

Such trigonometric polynomials always exist and an explicit formula can be found in \cite{nathan}:
\begin{equation}
\label{eq:formula1}
H_{C_i}(t)=\sum_{i=1}^N z_i w_{0,i}+\sum_{i=1}^N z'_i w_{1,i}(t-t_i),
\end{equation}
where 
\begin{align}
w_{0,i}(t)&=\left(1-V_i(t)\sin\left(\frac{t}{2}\right)\cos\left(\frac{t}{2}\right)\right)U_i(t),\\
w_{1,i}(t)&=2\sin\left(\frac{t}{2}\right)\cos\left(\frac{t}{2}\right)U_i(t),\\
U_i(t)&=\prod_{k\neq i}\left(\frac{\sin\left(\tfrac{t-t_k}{2}\right)}{\sin\left(\tfrac{t_i-t_k}{2}\right)}\right)^2,
\end{align}
\begin{equation}
V_i(t)=\sum_{k\neq i}2\cot\left(\frac{t_i-t_k}{2}\right).\label{eq:formula2}
\end{equation}


\noindent Here $N$ is the number of $\sigma$-crossings that strands from the component in question are involved in (counted with multiplicity) and the $t_i$, $i=1,2,\ldots,N$, are the positions of the $\sigma$-crossings. If for example such a crossing occurs between the strands $(C_1,j)$ and $(C_2,j')$ at $t\in[0,2\pi]$, then one of the $t_i$ for the component $C_1$ should be $\tfrac{t+2\pi (j-1)}{s_{C_1}}$ and one of the $t_k$ for the component $C_2$ should be $\tfrac{t+2\pi (j'-1)}{s_{C_2}}$. The values $z_i$ are the desired values of the interpolating trigonometric polynomials $H_{C_1}\left(\frac{t+2\pi (j-1)}{s_{C_1}}\right)=H_{C_2}\left(\frac{t+2\pi (j'-1)}{s_{C_2}}\right)$. The values $z'_i$ are the desired values of the derivatives $H_{C_1}$ and $H_{C_2}$ at the crossing. We only have to make sure that these values lead to the correct crossing signs as described above, so we have a lot of freedom in the actual choice of data values.

 
The interpolation for the radius functions $R_C$ is now comparatively simple. We only need to guarantee that at a $\sigma$-crossing the correct loop has a smaller radius, which is determined by the sign of the $\sigma$-crossing in question. We thus need two data points for each $\sigma$-crossing, one for each radius of the two strands involved in the crossing. If the strand $(C_1,j)$ passes through the hole of $(C_2,j')$, then $R_{C_1}\left(\tfrac{t+2\pi (j-1)}{s_{C_1}}\right)$ should be smaller than $R_{C_2}\left(\tfrac{t+2\pi (j'-1)}{s_{C_2}}\right)$, which can be achieved by choosing the data values at these points appropriately. We obtain interpolating trigonometric polynomials $R_C$. Since the radii must be positive at all times, we add the same sufficiently large constant to all $R_C$'s.

Since we have specified only very few values of $R_C$, there could be intersections at this stage. However, after multiplying every $R_C$ by the same small enough positive constant $\epsilon$ no intersections can occur. This is because the only points, where the $x$- and $y$-coordinates of two strands agree are at $\sigma$-crossings, which are precisely the points where we have specified the radii.

We claim that the loop braid that is parametrised by the constructed functions $F_C$, $G_C$, $H_C$ and $R_C$ is isotopic to the desired loop braid. To prove this we refer to Remark \ref{rem:equivloopsing}, which tells us that it suffices to compare the singular braids with signed intersection points corresponding to the core lines of the two loop braids. We call the singular braid with signed intersection points that we obtain via the interpolation procedure $B_1$ and the singular braid that corresponds to the centre lines of the desired loop braid $B_2$. We successively deform $B_1$ in the intervals $[\tfrac{2\pi k}{\ell},\tfrac{2\pi(k+1)}{\ell}]$, into $B_2$, first for $k=0$, then $k=1$ and so on. In each such interval $B_2$ has a unique crossing and the permutation of the strands of $B_1$ in that same interval is the same transposition. Suppose that $(C_i,j)$ is a strand of $B_1$ that does not correspond to a strand of $B_2$ that is involved in a $\sigma$-crossing in $[\tfrac{2\pi k}{\ell},\tfrac{2\pi(k+1)}{\ell}]$. Then, since it has the same $y$-coordinate value $y_{i,j,k}$ on all crossings that it is involved in (all of which are classical crossings), we can deform it so that its $y$-coordinate is constant $y_{i,j,k}$ throughout the interval $[\tfrac{2\pi k}{\ell},\tfrac{2\pi(k+1)}{\ell}]$. If $(C_i,j)$ does correspond to a strand that is involved in a $\sigma$-crossing of $B_2$ in $[\tfrac{2\pi k}{\ell},\tfrac{2\pi(k+1)}{\ell}]$ it can be deformed so that its $y$-coordinate is constant $y_{i,j,k}$ throughout the interval $[\tfrac{2\pi k}{\ell},\tfrac{2\pi(k+1)}{\ell}]$ except in a neighbourhood of the $\sigma$-crossing, where its $y$-coordinate takes values between $y_{i,j,k}$ and the $y_{i',j',k}$, where $(C_{i'},j')$ is the other strand involved in the $\sigma$-crossing. Note that these changes to the $y$-coordinates of the strands do not affect the braid word, since in our convention the word is read from the projection in the $xt$-plane (as in Fig. \ref{fig:loopgen}). Furthermore, the choice of $y_{i,j,k}$ guarantees that each strand has a different $y$-coordinate throughout the interval $[\tfrac{2\pi k}{\ell},\tfrac{2\pi(k+1)}{\ell}]$. 

Now we can deform each strand of $B_1$ that does not correspond to a strand that is involved in the unique crossing of $B_2$ in $[\tfrac{2\pi k}{\ell},\tfrac{2\pi(k+1)}{\ell}]$ by keeping its $y$-coordinate fixed and making its $x$-coordinate constant, too. This does not cause any intersections because all the strands have different $y$-coordinates. This leaves us with a singular braid with signed intersections, where in each $t$-interval $[\tfrac{2\pi k}{\ell},\tfrac{2\pi(k+1)}{\ell}]$ there are only two strands that are involved in crossings. Say the first of these crossings occurs at $t=t_k$. Since we know that the strands of $B_1$ are permuted in each interval $[\tfrac{2\pi k}{\ell},\tfrac{2\pi(k+1)}{\ell}]$ in the same way as the strands of $B_2$ and because any possible $\sigma$-crossing is the first crossing in the interval $[\tfrac{2\pi k}{\ell},\tfrac{2\pi(k+1)}{\ell}]$ by construction, we can apply the same argument that we used for the other strands in $[\tfrac{2\pi k}{\ell},\tfrac{2\pi(k+1)}{\ell}]$ for these two strands in the interval $[t_k+\epsilon,\tfrac{2\pi(k+1)}{\ell}]$ for some small $\epsilon>0$. Thus the strands of $B_1$ are deformed so that (just like for $B_2$) there is a unique crossing in $[\tfrac{2\pi k}{\ell},\tfrac{2\pi(k+1)}{\ell}]$. This crossing has the same type (classical or $\sigma$-crossing), the same sign and involves the same strands as the corresponding crossing of $B_2$. Applying these arguments successively to the intervals $[\tfrac{2\pi k}{\ell},\tfrac{2\pi(k+1)}{\ell}]$, $k=0,1,2,\ldots,\ell-1$, results in a singular braid that has only one crossing in each interval $[\tfrac{2\pi k}{\ell},\tfrac{2\pi(k+1)}{\ell}]$ and this crossing involves the same pair of strands as the unique crossing of $B_2$ in that same interval and with the same crossing sign. On the level of braid words the isotopies that were used above correspond to Eqs. (\ref{eq:rel1}) through (\ref{eq:rel2}). Hence, $B_1$ and $B_2$ are equivalent modulo Eqs. (\ref{eq:rel1}) through (\ref{eq:rel2}) and thus by Remark \ref{rem:equivloopsing} the interpolating trigonometric polynomials parametrise a loop braid that is isotopic to the desired one.
\end{proof}


This proves the following theorem and Theorem \ref{thm:intro1}.
\begin{theorem}
For every element $B$ of the loop braid group Algorithm 1 constructs a polynomial $f:\mathbb{R}^5\to\mathbb{R}^2$ such that the vanishing set $f^{-1}(0)\cap S^4$ on the unit 4-sphere contains a set that is ambient isotopic to the closure of $B$.
\end{theorem}

Like in \cite{bode:2016polynomial} we can give a bound on the degree. 
\begin{proposition}
\label{prop:bound}
Let $B$ be a loop braid on $s$ strands, represented by a word of length $\ell$ and $f:\mathbb{R}^5\to\mathbb{R}^2$ the corresponding polynomial found by Algorithm 1. For each component $C$ of the closure of $B$ we denote the number of strands in $C$ by $s_C$. Then
\begin{equation}
\deg f\leq2\sum_C\max\left\{\left\lfloor\frac{(s_C+1)(s_C\ell-1)+\ell s_C(s-s_C)-1}{2}\right\rfloor,s_C\right\}.
\end{equation}
\end{proposition}
\begin{proof}
The proof follows the same line of argument as the analogous result in \cite{bode:2016polynomial}. The degree of $f$ is determined by the degrees of the trigonometric polynomials, which in turn are determined by the number of data points in the corresponding interpolation problem. Since $F_C$ and $G_C$ are constructed in precisely the same way as in \cite{bode:2016polynomial}, i.e., require the same number of data points, their degrees satisfy the same bounds:
\begin{align}
\deg F_C&=\left\lfloor\frac{s_C\ell-1}{2}\right\rfloor,\nonumber\\
\deg G_C&\leq \left\lfloor\frac{(s_C+1)(s_C\ell-1)+\ell s_C(s-s_C)-1}{2}\right\rfloor.
\end{align}
Let $N$ be the number of $\sigma$-crossings that the component $C$ is involved in (counting with multiplicities). Then the trigonometric interpolation for $H_C$ requires $2N$ data points and the interpolation for $R_C$ requires $N$ data points.
Hence
\begin{align}
\deg H_C&=N,\\
\deg R_C&=\left\lfloor\frac{N-1}{2}\right\rfloor\label{eq:rbound}.
\end{align}   
In the same vain as in \cite{bode:2016polynomial} we find that
\begin{align}
\deg f=\sum_C \max\{2\deg F_C,2\deg G_C,\deg H_C,2\deg R_C,2s_C\}\leq2\sum_C\max\{\deg G_C,s_C\}.
\end{align}
Since $s_C$, $N$ and $\ell$ are positive integers, the bound for the degree of $G_C$ is the largest. The previously obtained bounds for the degrees of the trigonometric polynomials result in the bound given in the proposition.
\end{proof}

\begin{corollary}
Let $B$ be a loop braid on $s$ strands, represented by a word of length $\ell$, whose closure has only one component and let $f:\mathbb{R}^5\to\mathbb{R}^2$ be the corresponding polynomial found by Algorithm 1. Then
\begin{equation}
\deg f\leq2\left\lfloor\frac{(s+1)(s\ell-1)-1}{2}\right\rfloor.
\end{equation}
\end{corollary}

Bounds on the degree of trigonometric parametrisations of a given braid have also been studied in \cite{klaus1, klaus2}. The arguments in \cite{bode:2016polynomial} prove not only bounds on the degree of the constructed polynomials with knotted vanishing sets, but also on the degree of the trigonometric polynomials of such Fourier braids. Similarly, the arguments in the proof of Proposition \ref{prop:bound} offer bounds on the degree of trigonometric polynomials parametrising a given loop braid.

\begin{remark}
\label{rem:stereo}
Consider the inverse stereographic projection $\mathbb{R}^4\to S^4$ given by
\begin{align}
\label{eq:stereoo}
x&=\frac{2x_1}{x_1^2+x_2^2+x_3^2+x_4^2+1},\nonumber\\
y&=\frac{2x_2}{x_1^2+x_2^2+x_3^2+x_4^2+1},\nonumber\\
z&=\frac{x_1^2+x_2^2+x_3^2+x_4^2-1}{x_1^2+x_2^2+x_3^2+x_4^2+1},\nonumber\\
v&=\frac{2(x_3+\rmi x_4)}{x_1^2+x_2^2++x_3^2+x_4^2+1}.
\end{align}
A composition of this map with the constructed polynomial $f$ results in a rational map $\mathbb{R}^4\to\mathbb{R}^2$ whose denominator is some power of $x_1^2+x_2^2+x_3^2+x_4^2+1$, which never vanishes. Hence the numerator is a polynomial $\tilde{F}:\mathbb{R}^4\to\mathbb{R}^2$ whose vanishing set contains the closure of the desired loop braid $B$.

Note that it follows from the proof of Lemma \ref{lem:implicit} that the points on $B\subset S^4$ are regular points of $f$. Furthermore, the intersection of the vanishing set of $f$ and $S^4$ at any point on $B$ is transverse if $\lambda$ is sufficiently small. It follows that the points on $B\subset S^4$ are also regular points of $f|_{S^4}$. Hence the part of the vanishing set of $\tilde{F}$ that is isotopic to $B$ consists only of regular points. Again, this is completely analogous to the lower-dimensional situation in \cite{bode:2016polynomial}. However, zero is not necessarily a regular value, as the extra components might contain critical points. At this moment we are not aware of a constructive method that would remove the extra components or guarantee that zero becomes a regular value. Therefore, this remains an interesting question for future research.

Furthermore, since the points on the loop braid $B$ are regular points, any polynomial that differs from the construction by only a sufficiently small perturbation of the coefficients still has the closure of $B$ as part of its vanishing set. Hence we can assume that the coefficients are in $\mathbb{Q}[\rmi]$ and, after multiplying by a common denominator, in the Gaussian integers $\mathbb{Z}[\rmi]$.
\end{remark}

The algorithm outlined above follows the lower-dimensional case of \cite{bode:2016polynomial} closely. In particular, the polynomial $f_\lambda$ is obtained from the braid polynomial by substituting $v$ for $\rme^{\rmi t}$ and $\overline{v}$ for $\rme^{-\rmi t}$. In \cite{bode:quasipositive} the first author presented an alternative construction of polynomials for (classical) links. Substituting $v$ for $\rme^{\rmi t}$ and $\tfrac{1}{v}$ for $\rme^{-\rmi t}$ in the polynomial expression for $g_{\lambda}$ yields a rational map whose denominator is some power of $v$ and whose numerator is a holomorphic polynomial in complex variables $u$ and $v$. The numerator's vanishing set intersects $S^3$ transversally in a link that contains the desired link as a sublink. Hence, the holomorphicity comes at the price of additional components.

In the algorithm above we can just as well use this alternative substitution and obtain a rational map, whose denominator is some power of $v$ and whose numerator $\tilde{f}_{\lambda}$ is a polynomial in $x$, $y$, $z$ and the complex variable $v$. In particular, $\tilde{f}_\lambda$ is holomorphic with respect to $v$. All results that were shown for $f_\lambda$ in this section (with a different bound on the degree due to the degree of the denominator) remain true for sufficiently small values of $\lambda$ for $\tilde{f}_\lambda$.
\begin{proposition}
Let $B$ be a loop braid with $s$ strands, represented by a word of length $\ell$. For a component $C$ of its closure we denote the number of strands in $C$ by $s_C$. Then we can construct a polynomial $\tilde{f}_\lambda:\mathbb{R}^5\to\mathbb{R}^2$ with the following properties:
\begin{itemize}
\item $\tilde{f}_\lambda$ can be written as a polynomial $\mathbb{R}^3\times\mathbb{C}\to\mathbb{C}$. In particular, it is holomorphic with respect to the complex variable.
\item For small enough $\lambda>0$ the vanishing set of $\tilde{f}_\lambda$ intersects $S^4$ transversally in a set that contains the closure of $B$.
\item We have
\begin{align}\deg \tilde{f}_\lambda\leq &2\sum_C\max\left\{\left\lfloor\frac{(s_C+1)(s_C\ell-1)+\ell s_C(s-s_C)-1}{2}\right\rfloor,s_C\right\}\nonumber\\
&+2\sum_C \left\lfloor\frac{(s_C+1)(s_C\ell-1)+\ell s_C(s-s_C)-1}{2}\right\rfloor.
\end{align}
\end{itemize}
\end{proposition}

\subsection{An example}
\label{sec:example}

We illustrate the algorithm by going through the construction of a polynomial for the loop braid $\rho_1^{-1}\rho_2\sigma_1\rho_2\rho_1^{-1}$ in more detail. Note that this loop braid is obtained from the example braid $\sigma_1^{-1}\sigma_2\sigma_1^{-1}\sigma_2\sigma_1^{-1}$ in Section \ref{sec:review} by replacing each strand by a cylinder and changing one of the $\rho_1$s to a $\sigma_1$. A broken surface diagram is shown in Figure \ref{fig:examplesurf}, where the braid word is read from the bottom to the top. Note that for every value of the $t$-coordinate the intersection of the loop braid with the slice of constant $t$-values is the disjoint union of 3 (planar) loops. In particular, the link type does not change and $t|_B$ does not have any critical points. This might not be obvious from Figure \ref{fig:examplesurf}, since it displays a projection of a surface in 4-dimensional space to a 2+1-dimensional diagram, but it becomes apparent in comparison with Figure \ref{fig:loopgen}, which explains how broken surface diagrams can be interpreted as motions of loops. The fact that $t|_B$ does not have any critical points holds not only in this example but in general. This follows directly from the definition of the loop braid group, which does not allow motions that include intersections of a loop with itself or any other of the loops. The braid word in this example could be significantly simplified using the relations of the loop braid group, which would lead to a simpler polynomial, but this is not the point of this example. We simply want to illustrate how Algorithm 1 works for an arbitrary braid word.

\begin{figure}
\centering
\labellist
\large
\pinlabel $x$ at 830 480
\pinlabel $t$ at 300 1000
\pinlabel $y$ at 100 100
\endlabellist
\includegraphics[height=6cm]{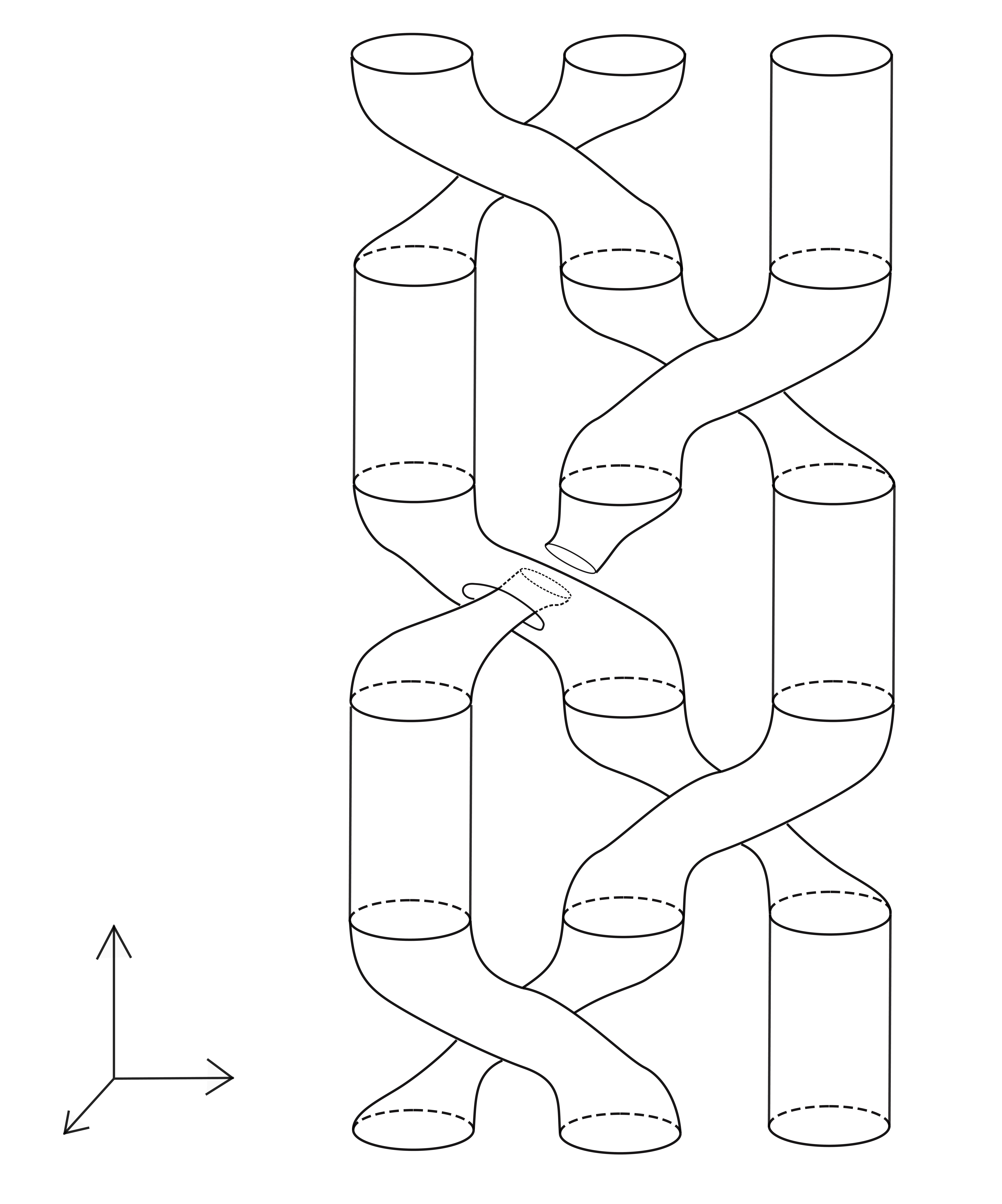}
\caption{A broken surface diagram of the example braid $\rho_1^{-1}\rho_2\sigma_1\rho_2\rho_1^{-1}$. \label{fig:examplesurf}}
\end{figure}

Note that the interpolation for the trigonometric polynomials $F_{C_1}$ and $F_{C_2}$, i.e., the $x$-coordinates of the centres of the rings, is the same as in the example in Section \ref{sec:review}. Hence we find the same functions
\begin{align}
F_{C_1}(t)=&-0.200+1.047\cos(t)+0.153\cos(2t),\nonumber\\
F_{C_2}(t)=&0.100+0.971\cos(t)-0.524\cos(2t)\nonumber\\
&-0.371\cos(3t)-0.076\cos(4t)-0.100\cos(5t).
\end{align}

We interpret the union of the graphs of $t\mapsto F_{C_i}\left(\tfrac{t+2\pi (j-1)}{s_i}\right)$, $[0,2\pi]\to\mathbb{R}$, $i=1,2$, $j=1,\ldots,s_i$, $s_1=1$, $s_2=2$, as a braid diagram without specified signs of crossings. The strand that forms the component $C_1$ is involved in 4 crossings. The values of the $t$-coordinate and the desired values of its $y$-coordinate at these crossings are
\begin{align}
\label{eq:datapointsc1}
&\{(t,y):\text{ crossing involving }C_1 \text{ at } t \text{ with desired value }y\}\nonumber\\
=&\{(0.794,-1),(1.857,1),(4.426,-1),(5.490,1)\},
\end{align} 
where we have chosen a $y$-value of 1 if $C_1$ is the overpassing strand in the $\rho_i$-crossing in question and -1 if it is the underpassing strand. The tuples in Eq. (\ref{eq:datapointsc1}) are the data points for the interpolation for $G_{C_1}$. Note that this strand is not involved in any $\sigma$-crossings. Hence it is the same interpolation as in Section \ref{sec:review} and we find the same trigonometric polynomial
\begin{equation}
G_{C_1}(t)=0.520-0.721\cos(t)-0.341\sin(t)+0.860\cos(2t)-1.243\sin(2t).
\end{equation} 

The two strands that form the component $C_2$ are involved in crossings at positions $t=0.794$, $t=3.142$, $t=4.426$ and $t=1.857$, $t=3.142$, $t=5.490$, respectively. These values are rounded, while the computation is done with numerical accuracy. We obtain the list of data points
\begin{align}
\{&(0.794/2,1),(3.142/2,0),(4.426/2,1),\nonumber\\
&(1.857/2+\pi,-1),(3.142/2+\pi,0),(5.490/2+\pi,-1)\}.
\end{align}
Note that this list differs from that in Eq. (\ref{eq:classdata}) only in that both values of the $y$-coordinate at the crossing at $t=3.142$ have been changed to 0, reflecting the change to a $\sigma$-crossing.

Using trigonometric interpolation we find
\begin{align}
G_{C_2}(t)=&0.047-0.057\cos(t)+0.804\sin(t)+0.047\cos(2t)-0.080\sin(2t)\nonumber\\
&-0.072\cos(3t)+0.804\sin(3t).
\end{align}

Since the component $C_1$ is not involved in any $\sigma$-crossings, there are no data points for its $z$- and $R$-coordinate. We can choose $H_{C_1}(t)=0$ and $R_{C_1}(t)=1$.

For the component $C_2$ there is a $\sigma$-crossing between the two strands that form this component at $t=3.142$. Thus the $z$-value of the rings should be equal at this crossing, i.e. we need the two data points $(3.142/2, 0)$ and $(3.142/2+\pi,0)$. Furthermore, the crossing in question is a $\sigma_1$ (as opposed to a $\sigma_1^{-1}$). Hence the derivative of the interpolating trigonometric polynomial $\tfrac{\partial H_{C_2}}{\partial t}$ should be smaller at $t=3.142/2$ than at $t=3.142/2+\pi$, since the ring that we regard as the first strand of $C_2$ is passing through the hole of the other strand from above. We choose $\tfrac{\partial H_{C_2}}{\partial t}(3.142/2)=-1$ and $\tfrac{\partial H_{C_2}}{\partial t}(3.142/2)=1$. The interpolation can now be performed by using Eqs. (\ref{eq:formula1}) through (\ref{eq:formula2}).
We obtain
\begin{equation}
H_{C_2}(t)=\cos(t).
\end{equation}

Similarly, for the radii of the rings in $C_2$, the only condition is that at the $\sigma$-crossing at $t=3.142$ the first strand of $C_2$ has a smaller radius than the second strand, so it can pass through without intersection points. Thus we can require that $R_{C_2}(3.142/2)=1$ and $R_{C_2}(3.142/2+\pi)=2$ and obtain as the interpolating trigonometric polynomial
\begin{equation}
R_{C_2}(t)=1.500-0.500\sin(t).
\end{equation}

It is usually necessary to add a positive constant to turn $R_{C}$ into a positive function. In this case, it is not necessary, since both functions are already positive.

Now we need to make sure that the rings never intersect. For this we find all values of $t$ for which there is a pair of strands that have the same $z$-coordinates. In this example this only happens for $t=\pi$, the position of the $\sigma$-crossing, where the $z$-coordinate of all strands is 0. The centres of the rings involved in the crossing are identical and the Euclidean distance from that point to the centre of the other ring (belonging to $C_1$) is 1.129. Recall that the radii of the rings determined by $R_{C_1}$ and $R_{C_2}$ are 1 and 2, respectively, for the strands of component $C_2$ and 1 for the strand belonging to $C_1$. Hence multiplying all $R_{C_i}$ by any constant less than $0.376$ results in trigonometric polynomials that do not induce any intersections, when taken as the parametrisations of the radii. Choosing the constant to be $\tfrac{1}{3}$ we obtain
\begin{align}
R_{C_1}(t)&=\frac{1}{3},\nonumber\\
R_{C_2}(t)&=\frac{1}{2}-\frac{1}{6}\sin(t).
\end{align} 


The trigonometric polynomials can be used to define $g_\lambda$ and $f_\lambda$, which are easily calculated (using appropriate software), but have too many terms to be reasonably displayed here.

\subsection{Other motion groups}
\label{sec:other}

The construction from the previous section yields polynomials for arbitrary elements of the loop braid group, the motion group of an $n$-component unlink in $\mathbb{R}^3$ without rotations of the components that change their orientations. We can adjust Algorithm 1 to accommodate certain motions of links that are split links, whose components are not trivial.

Let $B$ be a loop braid. For each component $C$ of its closure we can choose a (classical) braid $B_C$ in a solid torus $V$. Replacing a tubular neighbourhood of the loops in the component $C$ by $V$ containing $B_C$ yields an element of the motion group of the split link $L=\sqcup_C L_C^{s_C}$, where $L_C$ is the closure of $B_C$, $s_C$ is the number of strands in the component $C$ and $L_i^k$ denotes the split union of $k$ copies of a link $L_i$. Labelling the components of the closure of $B$ by $C_i$, $i=1,2,\ldots,n$, we denote the resulting element of the motion group by $\tilde{B}(B_{C_1},B_{C_2},\ldots,B_{C_n})$. It should be clear that in general this procedure does not generate the whole motion group, but it produces a subgroup of the motion group which is isomorphic to the loop braid group.

\begin{theorem}
Let $B$ be a loop braid on $s$ strands, represented by a word of length $\ell$, whose closure has $n$ components $C_i$, $i=1,2,\ldots,n$, consisting of $s_{C_i}$ strands. Let $B_i$, $i=1,2,\ldots,n$, be classical braids on $s_i$ strands and of length $\ell_i$. We denote the number of strands of a component $j$ of the closure of $B_i$ by $s_{i,j}$. Then we can construct a polynomial $f:\mathbb{R}^5\to\mathbb{R}^2$ with
\begin{align}
\deg f\leq&\sum_i \Bigg(2\sum_j\left(\max\left\{\left\lfloor\frac{(s_{i,j}+1)(s_{i,j}\ell_i-1)+\ell_i s_{i,j}(s_i-s_{i,j})-1}{2}\rfloor,s_{i,j}\right\}\right)\right.\nonumber\\
&\times \max\left\{\left\lfloor\frac{(s_i+1)(s_i\ell-1)+\ell s_i(s-s_i)-1}{2}\right\rfloor,s_i\right\}\Bigg)
\end{align}
and such that $f^{-1}(0)\cap S^4$ contains the closure of $\tilde{B}(B_{C_1},B_{C_2},\ldots,B_{C_n})$.
\end{theorem}
\begin{proof}
Composing an inverse stereographic projection $\mathbb{R}^3\to S^3$ with the polynomials obtained from the construction in \cite{bode:2016polynomial} we can construct polynomials $f_i:\mathbb{R}^3\to\mathbb{R}^2$ such that $f_i^{-1}(0)$ is the closure of $B_{C_i}$ and 
\begin{equation}
\deg f_i\leq 2\sum_j\max\left\{\left\lfloor\frac{(s_{i,j}+1)(s_{i,j}\ell_i-1)+\ell_i s_{i,j}(s_i-s_{i,j})-1}{2}\right\rfloor,s_{i,j}\right\}.
\end{equation}
By construction 0 is a regular value for all $i$ and the vanishing set can be assumed to lie in a tubular neighbourhood of the unit circle in the $xy$-plane.

We change the definition of the braid polynomial $g_\lambda$ in Eq. (\ref{eq:looppoly}) to
\begin{align}
g_{\lambda}(x,y,z,t)\defeq\prod_{i=1}^n&\prod_{j=1}^{s_{C_i}}\Bigg(\rho_{C_i,j}(t)^{\deg f_i}\nonumber\\
&\times f_i\left(\frac{x}{\rho_{C_i,j}(t)\lambda}-\frac{x_{C_i,j}(t)}{\rho_{C_i,j}(t)},\frac{y}{\rho_{C_i,j}(t)\lambda}-\frac{y_{C_i,j}(t)}{\rho_{C_i,j}(t)},\frac{z}{\lambda}-z_{C_i,j}(t)\right)\Bigg).
\end{align}
Everything else in the algorithm can be done as before. We find trigonometric polynomials $F_{C_i}$, $G_{C_i}$, $H_{C_i}$ and $R_{C_i}$ that parametrise the loop braid by $x_{C_i,j}(t)=F_{C_i}\left(\tfrac{t+2\pi (j-1)}{s_{C_i}}\right)$, $y_{C_i,j}(t)=G_{C_i}\left(\tfrac{t+2\pi(j-1)}{s_{C_i}}\right)$, $z_{C_i,j}(t)=H_{C_i}\left(\tfrac{t+2\pi(j-1)}{s_{C_i}}\right)$ and $\rho_{C_i,j}(t)=R_{C_i}\left(\tfrac{t+2\pi(j-1)}{s_{C_i}}\right)$. The bounds for these trigonometric polynomials found in the proof of Proposition \ref{prop:bound} are still valid. Hence the degree of the output polynomial $f_{\lambda}$ is bounded by
\begin{align}
\deg f_{\lambda}\leq &\sum_i\left(\deg f_i\times\max\left\{\left\lfloor\frac{(s_i+1)(s_i\ell-1)+\ell s_i(s-s_i)-1}{2}\right\rfloor,s_i\right\}\right)\nonumber\\
\leq&\sum_i \Bigg(2\sum_j\left(\max\left\{\left\lfloor\frac{(s_{i,j}+1)(s_{i,j}\ell_i-1)+\ell_i s_{i,j}(s_i-s_{i,j})-1}{2}\right\rfloor,s_{i,j}\right\}\right)\nonumber\\
&\times \max\left\{\left\lfloor\frac{(s_i+1)(s_i\ell-1)+\ell s_i(s-s_i)-1}{2}\right\rfloor,s_i\right\}\Bigg).
\end{align}
\end{proof}

\subsection{Time evolution of knotted fields}
\label{sec:physics}
Questions about the time evolution of knots and links in physical systems that were constructed from complex-valued maps, such as \textit{quantum knots} or in knots in magnetic fields have been raised in \cite{kauffman2} and \cite{synt}, respectively. It was since proven that every knot type (or in fact every link with a finite number of components) can arise as the union of some connected components of the nodal set of a quantum wavefunction that satisfies Schrödinger's equation for the harmonic oscillator \cite{daniel:berry} and the hydrogen atom \cite{daniel:coulomb}. The proofs are to a large part constructive and in particular enable us to find explicit solutions numerically using a computer.

It was shown in \cite{danielexistence} that we can realise any loop braid as the time evolution of the nodal set of a wavefunction in a fixed time interval, say $t=0$ and $t=2\pi$. It follows that the fields undergo a periodic time evolution, as observed in quantum revivals \cite{revival}. In fact, the results in \cite{danielexistence} go beyond the motions of loops that are described by motion groups such as the loop braid group in the previous sections, since they allow for wavefunctions whose nodal sets undergo reconnection events. In particular, the link type of the nodal set is allowed to change over time.

This section is meant as a speculation on the question, in how far the constructed polynomials from the previous sections can be used to describe the time evolution of physical systems, in analogy to the creation of knotted initial conditions of physical systems that is made possible by the lower-dimensional polynomial construction. Finding a way to manipulate the polynomials to yield functions that satisfy certain differential equations and still maintain the desired topological properties is likely to be extremely challenging and at this moment we are not aware of such a procedure.

In light of the results of \cite{danielexistence} the relevance of our construction to physicsts is perhaps limited. If at some point in the future it should be possible to use our polynomials for the creation of knotted configurations of physical systems that evolve with time as desired, it would allow a comparison of these solutions with the solutions obtained from \cite{danielexistence} and ideally a more direct relation between physical, topolgical and algebro-geometric quantities.

At this stage it is not clear if and how the constructed polynomials from the previous sections can be used to create knotted configurations of physical systems in $\mathbb{R}^3$ that evolve as desired. Polynomials typically do not describe physically relevant configurations, but they have the advantage that they are quite flexible in the sense that we can usually modify the polynomial in such a way that it satisfies any boundary condition at $\infty$. 

A first hint of how such a construction might work can be seen from the lower-dimensional case in the construction of optical vortex knots in scalar fields. In certain optical regimes the paraxial equation
\begin{equation}   
\label{eq:parax}
(\partial_x^2+\partial_y^2)\Psi+2\rmi k \partial_z\Psi=0,
\end{equation}
which is the $2+1$-Schrödinger equation with some wavenumber $k$, is a good approximation for one component of Maxwell's equations of a monochromatic light beam, such as a laser beam. The trefoil knot and the figure-eight knot have been constructed as vortex knots $\Psi^{-1}(0)$ of such optical scalar fields $\Psi$ using complex-valued polynomials, whose vanishing sets are the trefoil knot and the figure-eight knot respectively \cite{mark}. This is achieved by multiplying the polynomial by a Gaussian factor $\rme^{-(x^2+y^2)/w^2}$ with width $w$ and evolving the values of the resulting function in the $z=0$-plane according to Eq. (\ref{eq:parax}). It should be pointed out that while this procedure leads to optical fields whose vortex lines give the desired knot in the case of the trefoil knot and the figure-eight knot, this is not true in general. There are several polynomials for knots such as $5_2$ \cite{bode:52} for which so far we have not been able to create a corresponding optical field and at the moment it is not understood what determines if this method results in the desired vortex topology.

Maybe a similar phenomenon occurs in higher dimensions. Perhaps there is some way of modifying the constructed polynomial $f:\mathbb{R}^4\to\mathbb{R}^2$ (cf. Remark \ref{rem:stereo}), whose vanishing set contains the closure of the loop braid $B$, in such a way that evolving a 3-dimensional slice of this function according to the time-dependent Schrödinger equation results in a quantum-mechanical wavefunction, whose nodal lines evolve periodically as prescribed by the loop braid $B$. It should be clear from the vagueness of this paragraph that at the moment this procedure is still highly speculative. Future research will determine how much of this can be realised mathematically and experimentally. We would also like to point out that this treatment is not really quantum mechanical, since we treat the wavefunction as a classical object whose evolution is determined by the Schrödinger equation. We do not make any statements about energy eigenstates or how a loop braid could be measured or observed in such a system.

Also recall that the parameter $t$, which was regarded as time in the construction of the loop braid is mapped to a cyclic variable, e.g. $\arg(v)$. It is therefore advisable to consider our functions as polynomials on $\mathbb{R}^3\times S^1\subset\mathbb{R}^4$ instead of the space-time $\mathbb{R}^4$. It is not clear at the moment how this will affect any possible applications. It is imminent that with a periodic time variable, the physical systems undergo periodic time evolution, as observed in quantum revivals \cite{revival}.

The functions are constructed such that the vanishing set at any given time contains a split link of the same link type. In particular, reconnections that change the link cannot be constructed with our method. The class of functions is therefore naturally more restricted than in \cite{danielexistence}.

Another field of interest is the knotting of flow lines of electric and magnetic fields satisfying Maxwell's equations \cite{ranada, kedia1, kedia2}. In comparison to the quantum setting in the previous paragraphs, less is known here and in particular, we are not aware of any results that guarantee the existence of magnetic fields with closed flow lines that move as prescribed by a given path in the configuration space. The construction in \cite{kedia1} based on work by Bateman presently seems to be the only explicit construction that yields knotted flow lines whose link type (in this case a torus link) does not change over time.

An electromagnetic field is a time-dependent function $\mathbf{F}_t:\mathbb{R}^3\to\mathbb{C}^3$, whose real part $\mathbf{E}_t:\mathbb{R}^3\to\mathbb{R}^3$ is the electric field and whose imaginary part $\mathbf{B}_t:\mathbb{R}^3\to\mathbb{R}^3$ is the magnetic field. Together they satisfy Maxwell's equations.

We would like to indicate how our polynomials can be turned into time-dependent vector fields with flow lines that change as prescribed by a loop braid $B$.
\begin{proof}[\textbf{Proof of Proposition \ref{prop:vector}}]


In Section \ref{sec:loop} we have seen how for every loop braid $B$ we can generate a function $g_{\lambda}:\mathbb{R}^3\times[0,2\pi]\to\mathbb{C}$ whose vanishing set is equal to $B$ for any positive value of $\lambda$. In the following the scaling parameter $\lambda$ will not play a role and we are free to set it equal to 1. We can interpret the variable $t$ that runs through the interval $[0,2\pi]$ as a time variable, which turns $g\defeq g_1$ into a time-dependent complex-valued function on $\mathbb{R}^3$ that is $2\pi$-periodic with respect to the time variable $t$.
The next step follows ideas from Ra{\~n}ada \cite{ranada}. We define the vector field $V_t:\mathbb{R}^3\to\mathbb{R}^3$, $t\in\mathbb{R}$ as
\begin{equation}
V_t(x,y,z)\defeq\frac{1}{2\pi\rmi}\frac{\nabla g(x,y,z,t)\times\nabla\overline{g(x,y,z,t)}}{1+g(x,y,z,t)\overline{g(x,y,z,t)}},
\end{equation} 
where the gradient $\nabla$ is taken with respect to the three space variables $x$, $y$ and $z$.
It is a simple calculation that the fields $V_t$ have zero divergence for all values of $t$. The $2\pi$-periodicity of $V_t$ follows directly from that of $g$.
Furthermore, for every value of $t$ the flow lines of $V_t$ are tangent to the level sets of $g(\cdot,\cdot,\cdot,t):\mathbb{R}^3\to\mathbb{C}$. Since the vanishing set of $g$ for $t$ varying between $0$ and $2\pi$ is an unlink, whose components move according to $B$, the time-dependent vector field $V_t$ satisfies the conditions from the statement of the Proposition. 
\end{proof}

It was pointed out to us by a referee that the vector field $V_t$ can alternatively be defined without employing Ra{\~n}ada's work. Instead it uses the cross-product between three vectors in $\mathbb{R}^4$ and sets
\begin{equation}
V_t(x,y,z)\defeq \nabla \text{Re}(g)\times\nabla\text{Im}(g)\times\nabla t.
\end{equation}
Here $\nabla$ denotes the gradient with respect to the coordinates $(x,y,z,t)$. A priori the resulting vector is an element of $\mathbb{R}^4$, but since it is tangent to each time slice $\mathbb{R}^4\cap\{t=\text{const}\}$, we can interpret $V_t$ again as a time-dependent vector field on $\mathbb{R}^3$. A straightforward calculation establishes that this vector field also satisfies the properties stated in Proposition \ref{prop:vector}.

\begin{remark}
The proof of Proposition \ref{prop:vector} can be easily adopted to a construction of elements $\tilde{B}(B_1,B_2,\ldots,B_n)$ of other motion groups as introduced in Section \ref{sec:other}.
\end{remark}

Given previous work on time evolution of electromagnetic fields and that we have constructed a time-dependent vector field whose flow lines follow a prescribed periodic topological time evolution, we wonder if it could be possible to take the vector fields $V_t$ as the magnetic component $\mathbf{B}_t$ of an electromagnetic field. In other words, we are asking for another time-dependent vector field $\mathbf{E}_t$ such that $\mathbf{E}_t+\rmi V_t$ satisfies Maxwell's equations. Magnetic fields with knotted flow lines have been suggested as tools in the generation of fusion power \cite{fusion}.

Such a field $\mathbf{E}_t$ exists if and only if $V_t$ satisfies the wave equation
\begin{equation}
(\partial_{x_1}^2+\partial_{x_2}^2+\partial_{x_3}^2)V_t=-\partial_t V_t,
\end{equation}
where $\partial_x$ denotes the derivative with respect to a variable $x$. We have set all constants of nature equal to 1.

There is no reason to believe that the constructed vector fields $V_t$ satisfy this equation. The vector field that is obtained from the trigonometric polynomials in Section \ref{sec:example} for example does not. However, we would like to remind the reader that there was a lot of freedom in the construction, from the choice of braid representative to choosing data points for the trigonometric interpolation and topological invariance of the vanishing set of the constructed polynomial under small perturbations of the coefficients. It is not impossible that there is a set of choices that leads to vector fields $V_t$ that can be used as the magnetic part of an electromagnetic field. However, constructing $V_t$ such that it satisfies the wave equation remains a big challenge.

It should also be pointed out that the knotting in the constructed vector fields $V_t$ is in a sense weaker than that in \cite{kedia1, kedia2}. Our construction guarantees that we find flow lines of the same link type at any instance of time. Flow lines that correspond to other level sets of $g$ might change their link type over type periodically. This is in contrast to the ideas in \cite{kedia1, kedia2}, where the topology of all flow lines is conserved for all time. Note however that our construction discusses arbitrary links, not just torus links, and arbitrary movements of the different components, while the construction in \cite{kedia1, kedia2} allows no such control.

\section{Surface braids in $\mathbb{C}\times S^1\times S^1$}
\label{sec:2dim}

Section \ref{sec:loop} generalizes Algorithm 0 from classical braids to loop braids. This follows the interpretation of the braids as (equivalence classes of) motions of points in $\mathbb{C}=\mathbb{R}^2$. Another possibility is to regard braids as (equivalence classes of) covering maps.
\begin{definition}
A braid $B$ on $n$ strands is a subset of $\mathbb{C}\times [0,2\pi]$ such that 
\begin{itemize} 
\item the restriction of the projection map $p:\mathbb{C}\times [0,2\pi]\to[0,2\pi]$ to $B$ is a covering map of degree $n$ and
\item the boundary of $B$ is $X_n\times \partial[0,2\pi]=(X_n\times \{0\})\cup (X_n\times \{2\pi\})$, where $X_n$ is the set of $n$ fixed points in $\mathbb{C}$.
\end{itemize}
\end{definition}

This definition is equivalent to the usual definition in terms of geometric braids. It inspires the following definition of surface braids.
\begin{definition}
\label{def:surfbraid}
Let $F$ be a connected compact oriented surface. A surface braid of degree $n$ is a compact oriented surface $B$ in $\mathbb{C}\times F$ such that 
\begin{itemize}
\item the restriction of the projection map $p:\mathbb{C}\times F\to F$ to $B$ is an oriented simple branched covering map and
\item the boundary of $B$ is a trivial closed braid $X_n\times\partial F\subset \partial(\mathbb{C}\times F)$, where $X_n$ is the set of $n$ fixed points in $\mathbb{C}$.
\end{itemize}
\end{definition}

Fiber-preserving isotopies respecting the condition on the boundary divide the set of surface braids over $F$ in equivalence classes. In most cases the surface $F$ is taken to be a disk or (because of the boundary condition, equivalently,) the 2-sphere. Such surface braids were introduced by Viro. He and the second author also proved an analogue of Alexander's theorem, i.e., that every closed oriented surface embedded in $\mathbb{R}^4$ is ambient isotopic to a closed surface braid in $\mathbb{R}^4$ \cite{kamada:charact}. A more detailed overview on surface braids can be found in \cite{carter, kamada}. Rudolph \cite{rudolph832, rudolph85} studied surface braids over a disk without assuming the second condition of Definition \ref{def:surfbraid}.

We would like to generalize our algorithm to surface braids. However, it is not obvious at all how the kind of trigonometric parametrisation that has proven so useful in the classical case and in Section \ref{sec:loop} can be imitated here. Recall that previously the different coordinates of the strands were parametrised by functions such as $F_C\left(\tfrac{t+2\pi (j-1)}{s_C}\right)$, where $s_C$ is the number of strands in the component $C$ of the closure. If the surface braid that we want to construct has branch points there is no well-defined number $s_C$. We have to conclude that surface braids with branch points are out of reach for us at the moment. Thus, we here discuss surface braids without branch points. It is known that surface braids over the disk or the 2-sphere without branch points are equivalent to the trivial one, $X_n\times F$. Therefore, we will not consider the case of the disk or the 2-sphere here, but instead only study the case of the torus $F=S^1\times S^1$, where there are non-trivial surface braids without branch points. Such surface braids that are coverings of the torus have been studied by Nakamura \cite{Nakamura1, Nakamura2}. 

Since $S^1\times S^1$ does not have a boundary we can ignore the second condition in Definition \ref{def:surfbraid}. A \textit{trigonometric polynomial on} $S^1\times S^1$ is an element of the algebra generated by 
\begin{equation}
\{1,\cos(\varphi), \cos(\chi), \sin(\varphi), \sin(\chi),\cos(2\varphi), \cos(2\chi), \sin(2\varphi), \sin(2\chi),\ldots\},
\end{equation}
where $\varphi$ and $\chi$ are the coordinates on the first and second $S^1$-factor, respectively, taking values in $\mathbb{R}/2\pi\cong S^1$.
\subsection{The proof of Theorem \ref{prop:holosurf1}}
If every component $C$ of a surface braid $B$ on $S^1\times S^1$ can be parametrised in terms of trigonometric polynomials $F_C, G_C:S^1\times S^1\to\mathbb{R}$, then $g:\mathbb{C}\times S^1\times S^1\to\mathbb{C}$,
\begin{align}
\label{eq:gsurf}
g(u,\varphi,\chi)=\prod_{C_1}\prod_{C_2}\prod_{j=1}^{s_{C_1}}\prod_{k=1}^{s_{C_2}}&\left(u-\left(F_C\left(\frac{\varphi+2\pi (j-1)}{s_{C_1}},\frac{\chi+2\pi (k-1)}{s_{C_2}}\right)\right.\right.\nonumber\\
&\hspace{1cm}\left.\left.+\rmi G_C\left(\frac{\varphi+2\pi (j-1)}{s_{C_1}},\frac{\chi+2\pi (k-1)}{s_{C_2}}\right)\right)\right)
\end{align}
is a polynomial in $u$, $\rme^{\rmi \varphi}$, $\rme^{-\rmi \varphi}$, $\rme^{\rmi \chi}$ and $\rme^{-\rmi \chi}$ with $g^{-1}(0)=B$. The products are taken over the components $C_1$ of the closure of the classical braid $B\cap\{\chi=\text{constant}\}$ and the components $C_2$ of the closure of the classical braid $B\cap\{\varphi=\text{constant}\}$. The numbers $s_{C_1}$ and $s_{C_2}$ denote the numbers of strands in these components, respectively. Note that they are well-defined and independent of the constant values chosen for $\chi$ and $\varphi$, respectively, because there are no branch points.

By the Stone-Weierstrass theorem the trigonometric polynomials on $S^1\times S^1$ are dense in the set of continuous real-valued functions on $S^1\times S^1$. Hence any surface braid on $S^1\times S^1$ can be modified up to equivalence to a surface braid which is parametrised as above. Therefore, every braid leads to a polynomial $g$, which can be turned into a polynomial $f:\mathbb{C}^3\to\mathbb{C}$ in complex variables $u$, $v$ and $w$ by replacing $\rme^{\rmi \varphi}$ by $v$, $\rme^{-\rmi \varphi}$ by $\tfrac{1}{v}$, $\rme^{\rmi \chi}$ by $w$ and $\rme^{-\rmi \chi}$ by $\tfrac{1}{w}$ and multiplying by the common denominator, which is the product of some power of $v$ and some power of $w$. By construction we have
\begin{equation}
f^{-1}(0)\cap(\mathbb{C}\times S^1\times S^1)=B.
\end{equation}
We would like to particularly emphasize that $f$ is a complex polynomial and hence holomorphic.\qed

Theorem \ref{prop:holosurf1} is only referring to the existence of these functions, not to a concrete construction. This is because the density of trigonometric polynomial only guarantees the existence of a parametrisation of the desired form, but so far we have not been able to develop an algorithm that turns a given braid word into a parametrisation. There seems to be no obvious analogue of the interpolation methods used in \cite{bode:2016polynomial} and Section \ref{sec:loop}. In particular, we do not have an upper bound on the total degree of $f$. We can however focus on a special family of surface braids, for which trigonometric parametrisations are easily found.

\begin{definition}
\label{def:spinning}
Let $B$ be a classical braid on $s$ strands given by the parametrisation
\begin{equation}
\label{eq:classical}
\bigcup_{j=1}^s\left(X_j(\varphi)+\rmi Y_j(\varphi),\varphi\right), \qquad \varphi\in[0,2\pi].
\end{equation}
We say a surface braid $B'$ in $\mathbb{C}\times S^1\times S^1$ is a spinning braid (or is obtained by spinning $B$) if it is of the form
\begin{equation}
\label{eq:spinning}
\bigcup_{j=1}^s\left(\rme^{\rmi n\chi}(X_j(\varphi)+\rmi Y_j(\varphi)),\varphi,\chi\right), \qquad \varphi,\chi\in[0,2\pi]
\end{equation}
for some $n\in\mathbb{Z}$. In this case we denote $B'$ by $B(n)$. 
\end{definition}

\subsection{The proof of Theorem \ref{thm:intro2}}
It is easy to see that a spinning braid $B(n)$ is obtained from a classical braid $B$ by rotating it $n$ times along its core $\{(0,\varphi)|\varphi\in[0,2\pi]\}\subset\mathbb{C}\times[0,2\pi]$ as the second cyclic variable $\chi$ goes from 0 to $2\pi$. Eq. (\ref{eq:spinning}) provides us with a way to find a parametrisation of any spinning braid in terms of trigonometric polynomials. We use the tools from \cite{bode:2016polynomial} to find a parametrisation (\ref{eq:classical}) of any given braid $B$ in terms of trigonometric polynomials. Then Eq. (\ref{eq:spinning}) is a parametrisation of the spinning braid $B(n)$ in terms of trigonometric polynomials.

Thus $g:\mathbb{C}\times S^1\times S^1\to\mathbb{C}$ can be defined as in Eq. (\ref{eq:gsurf}), which leads to the polynomial $f:\mathbb{C}^3\to\mathbb{C}$ with $f^{-1}(0)\cap(\mathbb{C}\times S^1\times S^1)=B(n)$. Since we can find a Fourier parametrisation of any braid $B$ by trigonometric interpolation, we have a bound on the degree of the resulting function.

\begin{proposition}
\label{prop:bound2}
Let $B$ be a classical braid on $s$ strands with $\ell$ crossings in a diagram. Let $B(n)$ be a surface braid in $\mathbb{C}\times S^1\times S^1$ obtained by spinning $B$ for some $n\in\mathbb{Z}$. Then Algorithm 2 finds a holomorphic polynomial $f:\mathbb{C}^3\to\mathbb{C}$ in complex variables $u$, $v$ and $w$ such that
\begin{itemize}
\item $f^{-1}(0)\cap(\mathbb{C}\times S^1\times S^1)$ is equivalent to $B(n)$,
\item $\deg_u f=s$,
\item $\deg_w f=n$,
\item we have \begin{align}
\deg f\leq &2\sum_C\max\left\{\left\lfloor\frac{(s_C+1)(s_C\ell-1)+\ell s_C(s-s_C)-1}{2}\right\rfloor,s_C\right\}\nonumber\\
&+2\sum_C \left\lfloor\frac{(s_C+1)(s_C\ell-1)+\ell s_C(s-s_C)-1}{2}\right\rfloor+n.
\end{align}
\end{itemize}
\end{proposition} 

\noindent
\textbf{Algorithm 2:}
\begin{steps}
\item Perform the first two steps of Algorithm 0 for the classical braid $B$.
\item Define $g$ via
\begin{equation}
g(u,\varphi,\chi)=\prod_{C}\prod_{j=1}^{s_{C}}\left(u-\rme^{\rmi n\chi}\left(F_C\left(\frac{\varphi+2\pi (j-1)}{s_{C}}\right)+\rmi G_C\left(\frac{\varphi+2\pi (j-1)}{s_{C}}\right)\right)\right)
\end{equation}
\item Expand the product and replace $\rme^{\rmi \varphi}$ by $v$, $\rme^{-\rmi \varphi}$ by $\tfrac{1}{v}$ and $\rme^{\rmi \chi}$ by $w$ to obtain a rational map.
\item The numerator of the rational map is the desired polynomial $f$.
\end{steps}

This completes the proof of Theorem \ref{thm:intro2}.
\qed

In the 3-dimensional setting we have drawn special attention to the possibility of using our construction of polynomials to obtain explicit fibrations of link complements $S^3\backslash L$ over $S^1$. Having access to these explicit functions is useful in the study of some knotted fields in physics, where critical points of circle-valued maps often have some physical interpretation such as defect points in liquid crystals \cite{bode:52, km:2016topology} or stationary points of some fluid. The sole existence of fibration maps coming from polynomials plays a role in the construction of real algebraic links, the real analogue of Milnor's algebraic links, i.e., links of isolated singularities of real polynomials \cite{bode:realalg, milnor}.

While these applications might be unrelated to the 4-dimensional construction above, we would still like to point out that taking the argument of the polynomial obtained from Algorithm 2 results in a fibration map of the surface complement in $\mathbb{C}\times S^1\times S^1$ over $S^1$ if $n\neq 0$. This is because a critical point of $\arg f$ on $\mathbb{C}\times S^1\times S^1$ must also be a critical point of $\arg g$. These are given by the points where $\tfrac{\partial g}{\partial u}=0$, $\tfrac{\partial \arg g}{\partial \chi}=0$ and $\tfrac{\partial \arg g}{\partial \varphi}=0$. But at points with $\tfrac{\partial g}{\partial u}=0$, we have $\tfrac{\partial \arg g}{\partial \chi}=sn\neq0$.



\end{document}